\documentclass[11pt]{amsart}
\usepackage{graphicx}
\usepackage[mathcal]{euscript}
\usepackage{float}
\usepackage{hyperref}

\restylefloat{figure}

\newtheorem{Theorem}{Theorem}
\newtheorem{Corollary}[Theorem]{Corollary}

\newtheorem{theorem}{Theorem}[section]
\newtheorem{lemma}[theorem]{Lemma}
\newtheorem{corollary}[theorem]{Corollary}
\newtheorem{proposition}[theorem]{Proposition}

\newtheorem{scholium}[theorem]{Scholium}

\theoremstyle{definition}

\newtheorem{example}[theorem]{Example}

\theoremstyle{remark}

\newtheorem*{Remark}{Remark}

\numberwithin{equation}{section}

\newcommand{\ot}{\otimes}
\newcommand{\ra}{\rightarrow}

\newcommand{\BR}{\mathbb{R}}

\begin{document}

\title[Cube Diagrams]{Cube diagrams and $3$-dimensional Reidemeister-like moves for knots}

\author[S. Baldridge]{Scott Baldridge}
\author[A. Lowrance]{Adam Lowrance}

\thanks{This article was published in the Journal of Knot Theory and Its Ramifications, {\bf 21} (2012) no. 5, 1--39, DOI: 10.1142/S0218216511009832 \copyright~World Scientific Publishing Company.  An electronic version is available at www.worldscinet.com/jktr.}

\address{Department of Mathematics, Louisiana State University \newline
\hspace*{.375in} Baton Rouge, LA 70817, USA} \email{\rm{sbaldrid@math.lsu.edu}}

\address{Department of Mathematics, Vassar College \newline
\hspace*{.375in} Poughkeepsie, NY 12604, USA} \email{\rm{adlowrance@vassar.edu}}


\begin{abstract}
In this paper we introduce a representation of knots and links called a cube diagram.  We show that a property of a cube diagram is a link invariant if and only if the property is invariant under five cube diagram moves.  A knot homology is constructed from cube diagrams and shown to be equivalent to knot Floer homology.

\end{abstract}

\maketitle

\section{Introduction}

Reidemeister (1926) and Markov (1936) each described moves on 2-dimensional representations of links and proved when two representations were equivalent.  The structure of these representations and the limited number of moves makes these theorems useful for proving invariants.  In this paper we prove a 3-dimensional version of such a theorem.

\smallskip

A cube diagram is a representation of a piecewise linear embedding of a link into a 3-dimensional Cartesian grid.  Intuitively, this embedding can be thought of as a knot or link  in a $[0,n]\times [0,n] \times [0,n]$ cube (using $xyz$ coordinates) for some positive integer $n$ such that the link projections of the cube to each coordinate plane ($x=0$, $y=0$, and $z=0$) are grid diagrams (cf. Section 2).

\medskip

 A cube move  takes one cube diagram to another cube diagram and corresponds to a special ambient isotopy of the link.   There are two types of cube moves: cube commutation and cube stabilization moves  (cf. Section 3).   These moves are characterized by their projections to the three coordinate planes; they correspond to grid commutation and grid stabilization moves in each of the three projections.   All cube commutation moves and cube stabilization moves are generated up to symmetry of the cube diagram by the cube stabilization move in Figure~\ref{cubestab} and the four cube commutation moves in Figure~\ref{cubecom}.  Below is an example of a sequence of cube diagrams moves showing that the knot on the left is equivalent to the unknot.  

\begin{figure}[H]
\includegraphics{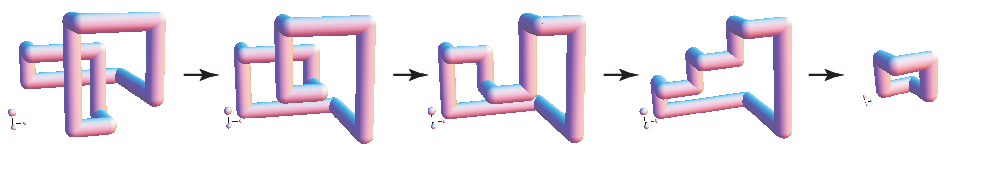}
\caption{The first three isotopies are cube commutation moves and the last is an application of two cube (de)stabilization moves.}
\end{figure}

Our main theorem is:

\medskip

\begin{Theorem}\label{maintheorem1}
Two cube diagrams represent ambient isotopic oriented links if and only if one can be obtained from the other by a finite sequence of cube commutation moves or cube stabilization moves.
\end{Theorem}

\medskip
As an
immediate corollary, we get a simple new way to check for knot and link invariants using the generating cube moves depicted in Figures \ref{cubestab} and \ref{cubecom}.

\medskip

\begin{Corollary}\label{maincorollary} Any property of a cube diagram that does not change under the 5 generating cube commutation and cube stabilization moves is an invariant of the link.
\end{Corollary}

\medskip

Our result is structurally different from Reidemeister's and Markov's in that we work directly with ambient isotopies of $\BR^3-L$. In creating cube diagrams, we were motivated by the search for a 3-dimensional data structure that was rigid enough to be able to easily define invariants, yet robust enough to represent all links and flexible enough that only a few types of moves were needed to transform one cube diagram of a link to any other cube diagram of the same link.  That study lead us to put strong conditions and symmetries on how $L$ is embedded in $\BR^3$, so strong in fact that the types and number of possible ambient isotopy moves become severely limited.  Thus, the conditions we impose in defining cube diagrams (marking and crossing conditions described in Section \ref{sec:cube_dia}) reduce the total number of generating isotopies to only 5 moves that take one cube diagram to another cube diagram of the same link.  


\medskip

There are other moves between piecewise linear  links such as triangle moves or lattice moves (\cite{DH}, cf. \cite{Prz}).
For example, the triangle move replaces a segment of a piecewise linear link with two additional segments if the three segments form a triangle and the segments and the interior of the triangle are disjoint from the link. The placement of the two additional segments is determined by the placement of the vertex incident with both segments. For any given piecewise linear embedding of a link, there are infinitely many different triangle moves that can be performed. Cube moves, on the other hand, differ from triangle moves and other piecewise linear moves because there are only finitely many instances of commutation and stabilization cube moves that can be performed on any given cube diagram. Hence, one has the desired control similar to that of Reidemeister or Markov moves, but in $3$ dimensions instead of $2$.

\medskip

While cube diagrams project to grid diagrams, grid diagrams rarely lift to cube diagrams. In a recent note, the first author and McCarty \cite{note} show that, for example, only about 20\% of size 8 grid diagrams of nontrivial knots  lift to cube diagrams, and that this percentage decreases as the size of the grid increases.  The sparsity of cube diagrams compared to grid diagrams is advantageous: cube diagrams can be used to develop strictly stronger invariants than grid diagrams. For example, the {\em cube number} of a link is the minimum size cube diagram that represents that link. Similarly, the {\em grid number} (or {\em arc index}) of a link is the minimum size grid diagram that represents that link. The grid number of a link is equal to the grid number of its mirror image. McCarty \cite{private} has shown that cube number can distinguish a knot from its mirror image while grid number cannot. For example, the cube number of the left-handed trefoil is five while the cube number of the right-handed trefoil is seven. Additionally, McCarty \cite{Mc} showed that a Legendrian version of cube number can distinguish between Legendrian knots with the same underlying knot type.

\medskip


\medskip

As an example of an invariant coming from cube diagrams, we present a  filtered homology theory for a link $L$.  The chain complex $C^-_y(\Gamma)$ of the homology theory is generated by cube states---certain configurations of lattice points on the $3$-dimensional Cartesian grid.  Each cube state has an associated Maslov and Alexander grading. The differential $\partial^-_y$ counts the number of empty cylinders between two cube states, and decreases the Maslov grading by one.  It is then easy to check that $(\partial^-_y)^2=0$, making $(C^-_y(\Gamma),\partial^-_y)$ into a filtered chain complex.  See Section~\ref{cubehomology} for definitions.  The homology of $(C^-_y(\Gamma),\partial^-_y)$ is invariant under cube stabilization and cube commutation moves, making it an invariant of the link by Corollary~\ref{maincorollary}.  We prove:

\medskip

\begin{Theorem}
\label{maintheorem2} Let $\Gamma$ be a cube diagram representing $L$ and $CH^-(L)$ be the homology of the complex $(C_y^-(\Gamma),\partial^-_y)$. Then
$$CH^-(L)\cong HFK^-(L)\otimes HFK^-(L),$$
where $HFK^-(L)$ is the ``minus" version of knot Floer homology with $\mathbb{Z}/2\mathbb{Z}$ coefficients.
\end{Theorem}

\medskip

The organization of this paper is as follows: In Section 2 we give the definition of a cube diagram and show that every link can be represented by a cube diagram.  In Section 3 we describe cube stabilization and commutation moves and prove Theorem~\ref{maintheorem1} using them.  Cube homology theory is developed in Section 4 and is proved to be equivalent to knot Floer homology.


\medskip
\section{Cube Diagrams}
\label{sec:cube_dia}
\medskip

\subsection{Definition}
Let $n$ be a positive integer and $C=[0,n]\times [0,n]\times [0,n] \subset \mathbb{R}^3$ thought of as $3$-dimensional Cartesian grid, i.e., a grid with integer valued vertices.  The number $n$ is called the {\em cube number}.  A \textit{flat} is any cuboid (a right rectangular prism) with integer vertices in $C$ such that there are two orthogonal edges of length $n$ with the remaining orthogonal edge of length $1$. Each edge of a flat must be parallel to one of the coordinate axes. A flat with an edge of length 1 that is parallel to the $x$-axis, $y$-axis, or $z$-axis is called an {\em $x$-flat}, {\em $y$-flat}, or {\em $z$-flat} respectively.

\medskip

Next we describe a way to specify an embedding of a link in the cube $C$.    A marking is a point with half integer coordinates in $C$ labeled by either an $X$, a $Y$, or a $Z$.  For each ${1 \times 1 \times 1}$ cube with integer vertices in $C$ assign either zero or one marking.  The set of markings must satisfy the following {\em marking conditions}:

\begin{itemize}
    \item each flat has exactly one $X$, one $Y$, and one $Z$ marking;\\

    \item the markings in each flat form a right angle such that each ray is parallel to a coordinate axis;\\

    \item for each $x$-flat, $y$-flat, or $z$-flat, the marking that is the vertex of the right angle is an $X, Y,$ or $Z$ marking respectively.
\end{itemize}

\medskip

Denote the set of $X$'s in $C$ by $\mathcal{X}$. Similarly define $\mathcal{Y}$ and $\mathcal{Z}$. Note that the cube $C$  is canonically oriented by the standard orientation of $\BR^3$ (right hand orientation).


\medskip

\begin{Remark}  An alternate but equivalent formulation of the marking conditions is given by the first author in a recent paper \cite{Bald}.  In the alternate formulation, the third condition is changed by a simple permutation of the letters ($X \mapsto Z$, $Y\mapsto X$, $Z\mapsto Y$).
\end{Remark}

\medskip

Embed an oriented piecewise linear link $L$ into $C$ by connecting pairs of markings with a line segment whenever two of their corresponding coordinates are the same.   Each segment is oriented to go from either $X$ to  $Y$,  $Y$ to $Z$, or from $Z$ to $X$. The markings in each flat define two perpendicular segments of the link $L$ joined at a vertex, call the union of these segments a {\it cube bend}. If a cube bend is contained in an $x$-flat, we call it an {\it $x$-cube bend}. Similarly, define {\it $y$-cube bends} and {\it $z$-cube bends}.

\medskip

Arrange the markings in $C$ so that the following {\em crossing conditions} hold:
\begin{itemize}
\item At every intersection point of the $(x,y)$-projection of $L$, the segment parallel to the $x$-axis has smaller $z$-coordinate than the segment parallel to the $y$-axis.\\

\item At every intersection point of the $(y,z)$-projection of $L$, the segment parallel to the $y$-axis has smaller $x$-coordinate than the segment parallel to the $z$-axis.\\

\item At every intersection point of the $(z,x)$-projection of $L$, the segment parallel to the $z$-axis has smaller $y$-coordinate than the segment parallel to the $x$-axis.
\end{itemize}

\medskip

If the cube $C$, the markings $\{\mathcal{X},\mathcal{Y},\mathcal{Z}\}$, and the link $L$ satisfy the marking and crossing conditions, then define $\Gamma = (C, \{\mathcal{X},\mathcal{Y},\mathcal{Z}\}, L)$ to be a {\em cube diagram} representing the (oriented) link $L$.    When the context is clear, we will often use the term cube diagram to describe either the triple or the embedded link it represents.

\medskip

\begin{figure}[H]
\includegraphics[scale=1]{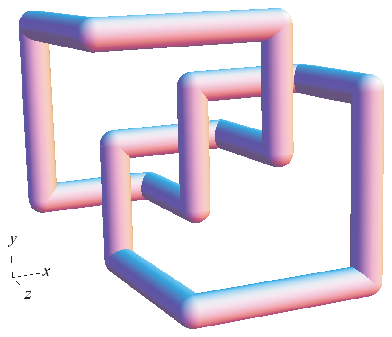}
\caption{An example of the trefoil represented by a cube diagram.  Computer programs are available for drawing and computing invariants of cube diagrams \cite{prog1, prog2}.} \label{trefoil}
\end{figure}

\medskip

\subsection{Grid Diagrams}

Cube diagrams have a nice relationship with  grid diagrams, combinatorial structures that encode information about link projections.   Grid diagrams were introduced by Brunn~\cite{Brunn} over a 100 years ago, and discussed more recently in Cromwell~\cite{Cromwell} and Dynnikov~\cite{Dyn}.  A {\it grid diagram} as it is usually defined in the literature is an $n\times n$ subset of the Cartesian grid such that each cell contains either zero or one marking.  A marking is a point with half integer coordinates labeled by either an $X$ or an $O$. The markings must satisfy the following conditions:

\medskip

\begin{itemize}

\item each row contains exactly one marking labeled $X$ and exactly one marking labeled $O$, and\\

\item each column contains exactly one marking labeled $X$ and exactly one marking labeled $O$.
\end{itemize}

\medskip

The grid diagram specifies an oriented link projection by drawing horizontal line segments in each row from the $O$ to the $X$ and vertical line segments in each column from the $X$ to the $O$. Moreover, at intersection points, the vertical segment are over-crossings. Let $\mathbb{X}$ and $\mathbb{O}$ denote the set of $X$ markings and $O$ markings respectively.

\medskip

We need an equivalent definition of a grid diagram that takes into account the orientation of the grid.  Given a $[0,n]\times [0,n] \subset \BR^2$ Cartesian grid with  labeled axes, define an orientation on the grid diagram by choosing an ordering of the axes: then ``rows'' and ``horizontal segments'' in the definition above are rows and segments parallel to the first axis chosen, and  ``columns'' and ``vertical segments'' are rows and segments parallel to the second axis chosen.  At each intersection point, the segment parallel to the second axis chosen is the over-crossing.  The orientation of the link itself is defined similarly.

\medskip

For example, the first picture in Figure~\ref{or} shows the  $(x,y)$-orientation and the second picture shows the $(y,x)$-orientation for a grid diagram.  Notice that changing the orientation of a  grid diagram changes the link to the reverse of its mirror.

\begin{figure}[H]
\includegraphics[scale=.3]{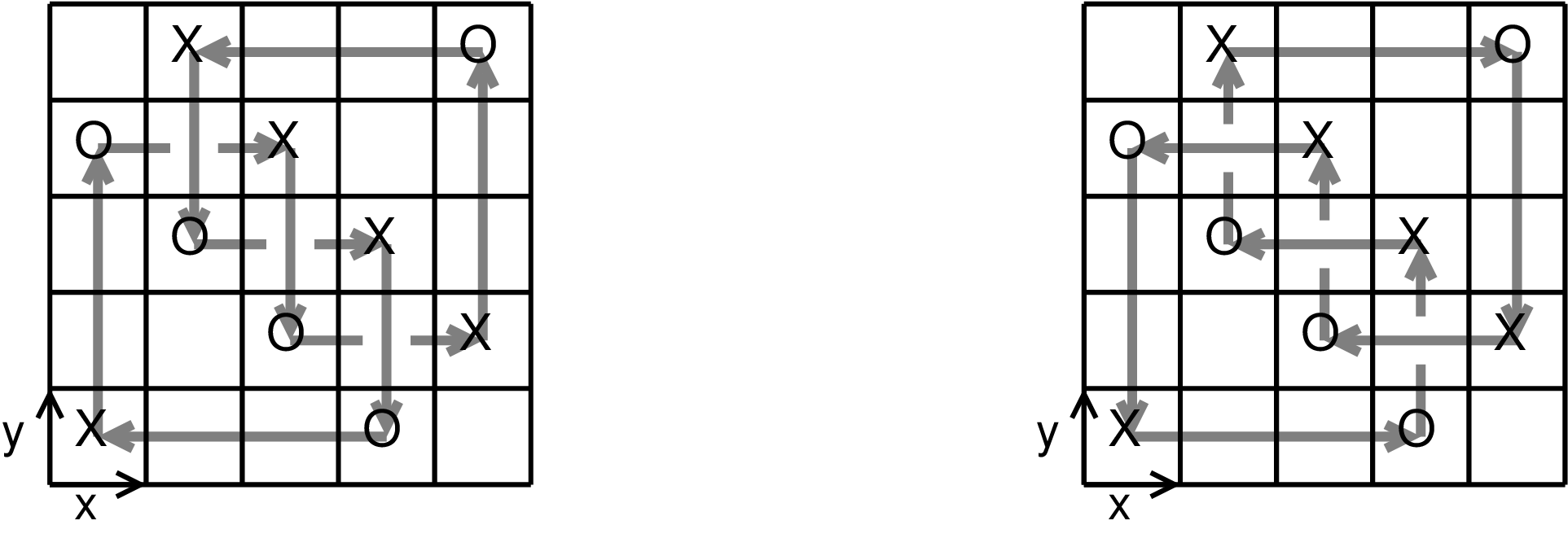}
\caption{The figure on the left is an $(x,y)$-oriented grid diagram and the figure on the right is a grid diagram with the same markings but with the $(y,x)$-orientation.} \label{or}
\end{figure}

\medskip

A cube diagram can be naturally projected onto an oriented grid diagram in a way that respects orientation, over-crossings, and markings.  As an example we show how to project the link in a cube diagram onto the $(x,y,0)$-plane.  Cube diagrams are populated by three types of markings ($X,Y,$ and $Z$) while grid diagrams only have two ($X$ and $O$); there is a preferred way of identifying the markings on the cube with the markings in the grid.  In this example, the $Z$ markings in the cube are projected to $X$ markings in the grid.   The other two markings, $X$ and $Y$, both get projected to $O$ markings in the grid.  (In general, for projections to the $(0,y,z)$-plane, the $X$ markings in the cube are projected to $X$ markings in the grid, and  for projections to the $(x,0,z)$-plane, the $Y$ markings in the cube are projected to $X$ markings in the grid.  In either case, the other markings in the cube are then projected to $O$ markings in the grid.)  With these identifications, notice that the conditions on the link in the cube and orientation of the link in the cube both correspond to the grid diagram with the $(x,y)$-orientation.  Call this identification of markings a {\em projection} of the cube diagram onto the $(x,y)$-oriented grid diagram and say that it is the {\em $(x,y)$-projection} of a cube diagram.

\begin{proposition}
\label{projection} A cube diagram projects onto the $(x,y)$-oriented grid diagram, the $(y,z)$-oriented grid diagram, and the $(z,x)$-oriented grid diagram.  Furthermore, the orientations of the grid diagrams are inherited from the orientation of the cube diagram.
\end{proposition}

\medskip

Figure~\ref{gooddia} shows the different projections of a cube diagram of an unknot.

\begin{figure}[h]
\includegraphics[scale=.3]{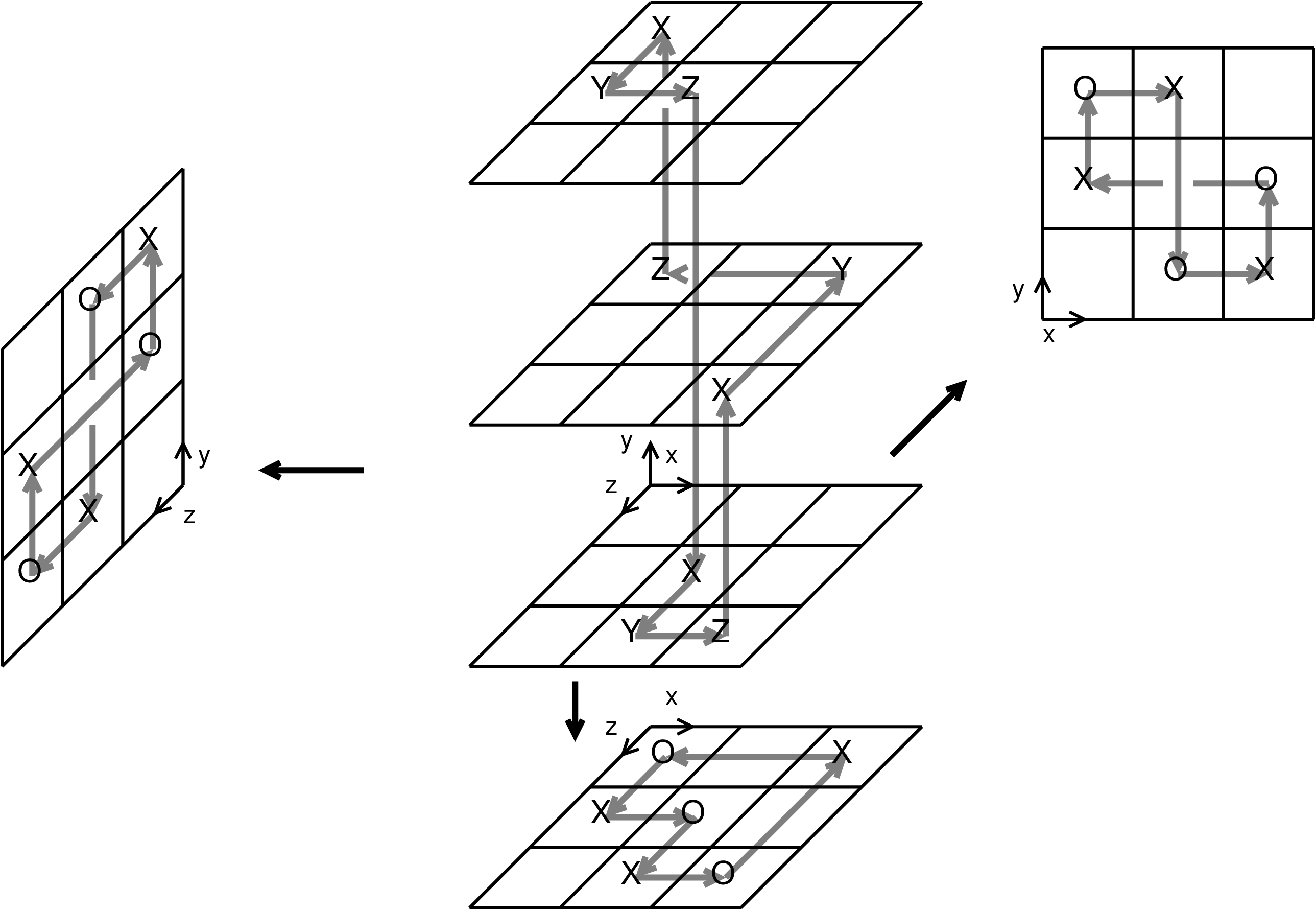}
\caption{A cube diagram can be depicted by stacking flats on top of one other, or by the projected $(x,y)$, $(y,z)$, and $(z,x)$-oriented grid diagrams.  Notice that it is possible to recover the cube diagram from any two of the projected grid diagrams.} \label{gooddia}
\end{figure}

\medskip

A corollary of the proposition above relates grid numbers to cube numbers of a link.  The {\em grid number of a link} is the minimum grid number over all grid diagrams of that link.  Similarly, the {\em cube number of a link} is the minimum cube number over all cube diagrams of the link.

\begin{corollary} The grid number of a link is less than or equal to the cube number of that link.
\end{corollary}

McCarty \cite{private} has shown that for some links (including the right-handed trefoil) the inequality is strict.


\medskip
\subsection{Existence} In this section we show that every link can be represented by a cube diagram. An efficient way to produce a cube diagram is to show that a grid diagram gives rise to a cube diagram.  But not every grid diagram does so:

\begin{example} There exists grid diagrams which cannot be the projections of any cube diagram. \label{ex:badgriddia}
\end{example}

To describe the example we need a bit of terminology.  A {\it bend} in a grid diagram for a link $L$ is a pair of segments in $L$ which meet at a common $X$ or $O$ marking (see Figure \ref{bend}).
\begin{figure}[H]
\includegraphics[scale=.3]{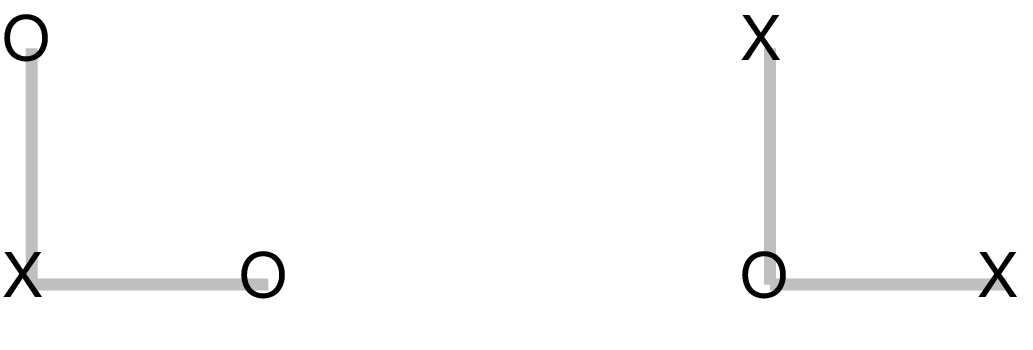}
\caption{An example of a bend in a grid diagram.} \label{bend}
\end{figure}
Each link component is the union of non-overlapping bends.  For a component of the link there are two partitions of non-overlapping bends depending on whether the two segments in each bend intersect in  an $X$ or $O$ marking.  Let $\mathcal{B}$ be a collection of such partitions, one partition for each component of the link (e.g., pick all bends that intersect at an $X$ marking). Then $|\mathcal{B}|$ equals the grid number of $G$. If a bend $b\in\mathcal{B}$ passes over some other segment of $L$ and passes under some other segment of $L$, then call $b$ {\it twisted}.

\medskip

\begin{figure}[H]
\includegraphics[scale = .4]{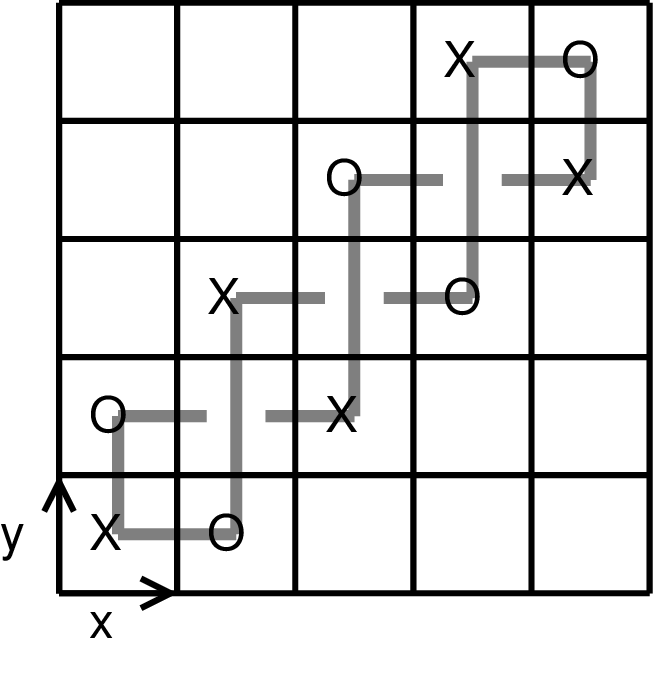}
\caption{This grid diagram cannot be a projection of a cube diagram.} \label{badgrid}
\end{figure}

The example is the grid diagram pictured in Figure~\ref{badgrid}.  Give the above grid diagram the $(x,y)$-orientation.  There are exactly two partitions of this grid diagram of the unknot into a set of bends. For either partition, the bends in the partition must correspond to the set of $z$-cube bends in the cube.  Note that for both partitions there is a pair of twisted bends $b_1$ and $b_2$ that intersect each other twice.  So if there existed a cube diagram for this partition that projected to the grid diagram, $b_1$ and $b_2$ would be $z$-cube bends in that cube, and therefore the $z$-coordinate for $b_1$ would both be greater than and less than the $z$-coordinate for $b_2$, which is a contradiction.

\medskip

However, a grid diagram with no twisted bends does potentially come from a projection of a cube diagram.  Such a grid diagram can be found from a given grid diagram by stabilizing.

\begin{lemma}\label{nobendtwistedlemma}
Every link $L\subset S^3$ has a grid diagram in which no bend is twisted.
\end{lemma}

\begin{proof} Let $G$ be a grid diagram for $L$, and let $\mathcal{B}$ be the set of bends for $G$. If $b\in\mathcal{B}$ is a twisted bend, then replace $G$ with $G^\prime$, where $G^\prime$ is obtained from $G$ by a stabilization at the $X$ marking of $b$. This stabilization adds one to the size of the grid diagram, but breaks $b$ into two different bends, neither of which is twisted. See Figure \ref{bendmove}.
\end{proof}
\begin{figure}[H]
\includegraphics[scale=.3]{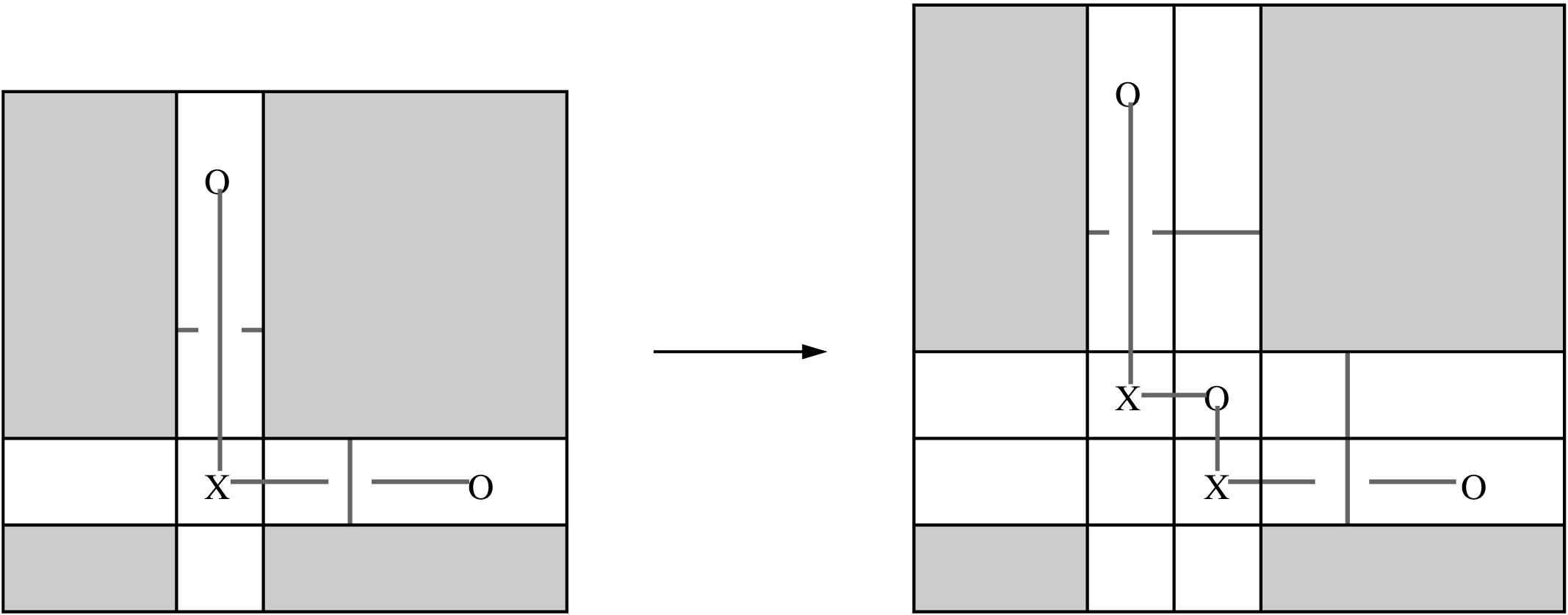}
\caption{Replace the twisted bend on the left with the stabilized picture on the right.} \label{bendmove}
\end{figure}

If $G$ is a grid diagram where no bend is twisted, then $\mathcal{B}$ can be naturally partitioned into three sets: bends which pass over other arcs, bends which do not contain any crossings, and bends which pass under other arcs, denoted $\mathcal{B}_{\mathrm{over}},\mathcal{B}_{\mathrm{neutral}},$ and $\mathcal{B}_{\mathrm{under}}$ respectively.

\medskip

\begin{proposition}
Every link $L$ in $S^3$ can be represented by a cube diagram.
\end{proposition}

\medskip

\begin{proof}
Let $G$ be a grid diagram for $L$ where no bend is twisted. Thus $\mathcal{B}$ can be partitioned into $\mathcal{B}_{\mathrm{over}},\mathcal{B}_{\mathrm{neutral}},$ and $\mathcal{B}_{\mathrm{under}}$. We embed the link into the cube such that the $(x,y)$-projection is exactly $G$. Transform each bend in $\mathcal{B}$ into a $z$-cube bend by placing the bend into a flat of the same dimensions as $G$. Then stack the flats containing the $z$-cube bends together so that the $z$-cube bends from $\mathcal{B}_{\mathrm{over}}$ have greater $z$-coordinate than $z$-cube bends from $\mathcal{B}_{\mathrm{under}}$. The $z$-cube bends from $\mathcal{B}_{\mathrm{neutral}}$ may be stacked anywhere. This stacking results in an embedding of the link in the cube with a  projection to the $(x,y)$-oriented grid diagram that agrees with the grid diagram $G$. 

\medskip

By construction, all the crossings in the $(x,y)$-projection meet the necessary conditions to be a cube diagram. However, the embedding of $L$ may not satisfy the crossing conditions for the other two projections.  Suppose that in the $(y,z)$-projection the segment parallel to the $y$-axis has greater $x$-coordinate than the segment parallel to the $z$-axis, as depicted in the first picture in Figure~\ref{crosssurgery} (violating the crossing conditions in the definition of a cube diagram). Then for that crossing, alter the cube diagram by increasing the cube diagram number by two and changing the local picture around the crossing to the second picture   depicted in Figure~\ref{crosssurgery}. The new cube diagram satisfies the crossing conditions for this crossing and does not add any new crossings in the other projections.  Other crossing condition violations in the $(y,z)$-projection are dealt with similarly.
\begin{figure}[H]
\includegraphics[scale=.4]{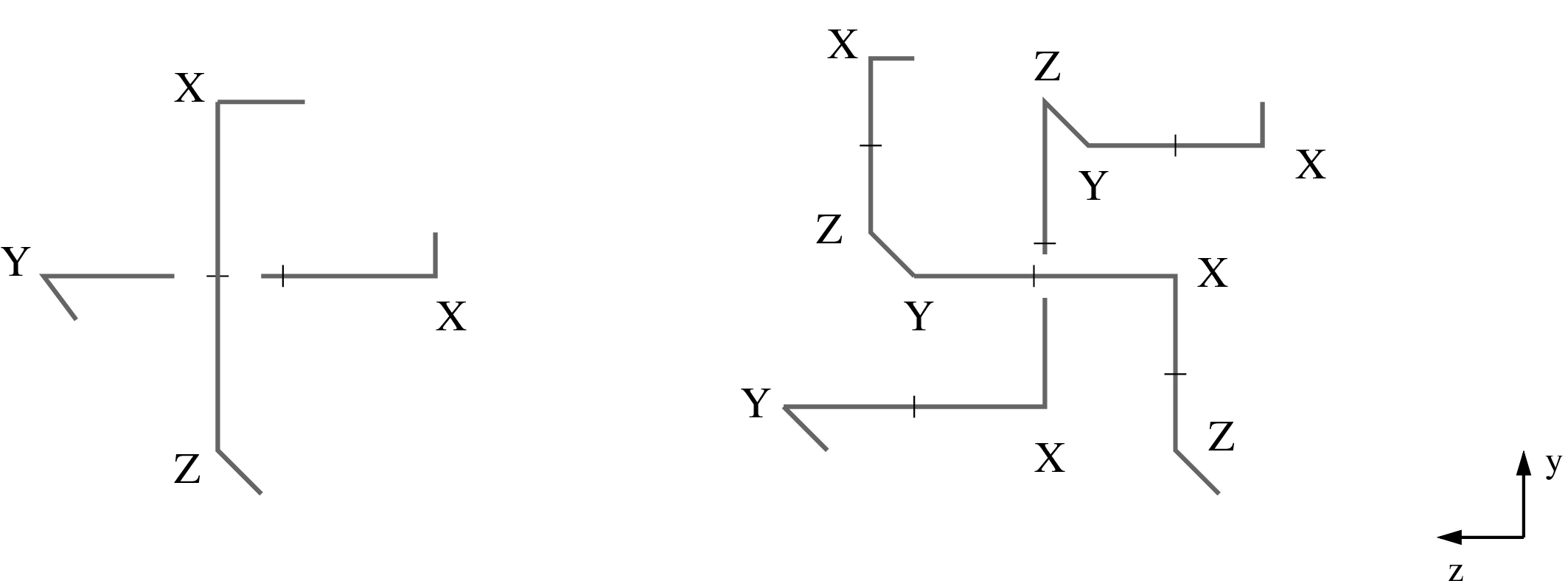}
\caption{Replace the crossing on the left with the new one on the right. The hashmarks along the link indicate line segments of length one.} \label{crosssurgery}
\end{figure}
Repeating this process for the $(z,x)$-projection produces a cube diagram $\Gamma$ for $L$.
\end{proof}


\medskip
\section{Invariant Moves}
\medskip
\label{invariantmoves}

First we describe invariant moves for grid diagrams, and then show corresponding moves for cube diagrams.

\medskip

\subsection{Invariant Grid Moves}
Cube diagrams represent oriented links. Clearly, each  link can have many different cube diagrams.  In this section we prove that any two grid diagrams corresponding to the same oriented link can be connected by a sequence of the following two elementary moves:

\smallskip

\begin{enumerate}
\item {\bf Stabilization.} A stabilization move adds an extra row and column to a grid diagram. Suppose $G$ is an $(x,y)$-oriented grid diagram with grid number $n$ and label the markings of $G$ by $\{X_j\}_{j=2}^{n+1}$ and $\{O_j\}_{j=2}^{n+1}$. Let $X_i$ and $O_i$ be two markings in a row $r$. Split $r$ into two new rows, with $X_i$ in one new row and $O_i$ in the other so that the $x$-coordinate of both markings remain the same as in G. Also add a new column to the diagram adjacent to the $X_i$ or the $O_i$  such that it is between the two markings, and then place two new markings $X_1$ and $O_1$ in the column so that $X_1$ occupies the same row as $O_i$ and $O_1$ occupies the same row as $X_i$. Destabilization is the inverse of this process.  See Figure \ref{stab}.
\begin{figure}[H]
\includegraphics[scale=.6]{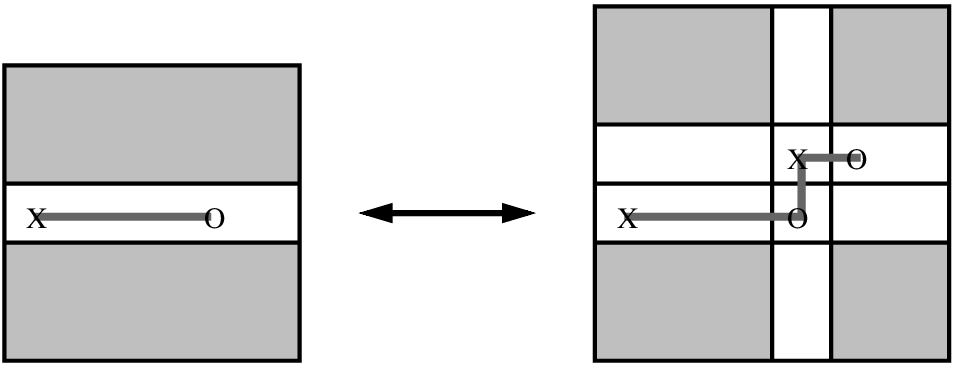}
\caption{{\bf Stabilization.} The segment on the left is replaced by the three segments on the right, which increases the size of the grid diagram by one.} \label{stab}
\end{figure}
\item {\bf Commutation.} Consider two adjacent rows in the grid diagram. In each row, project the line segment connecting the $X$ and $O$ to the $x$-axis. If the projections of the segments are disjoint, share exactly one point, or if the projection of one segment is entirely contained in the projection of the other, then the two rows can be interchanged. There is a similar move for columns.
\begin{figure}[H]
\includegraphics[scale=.6]{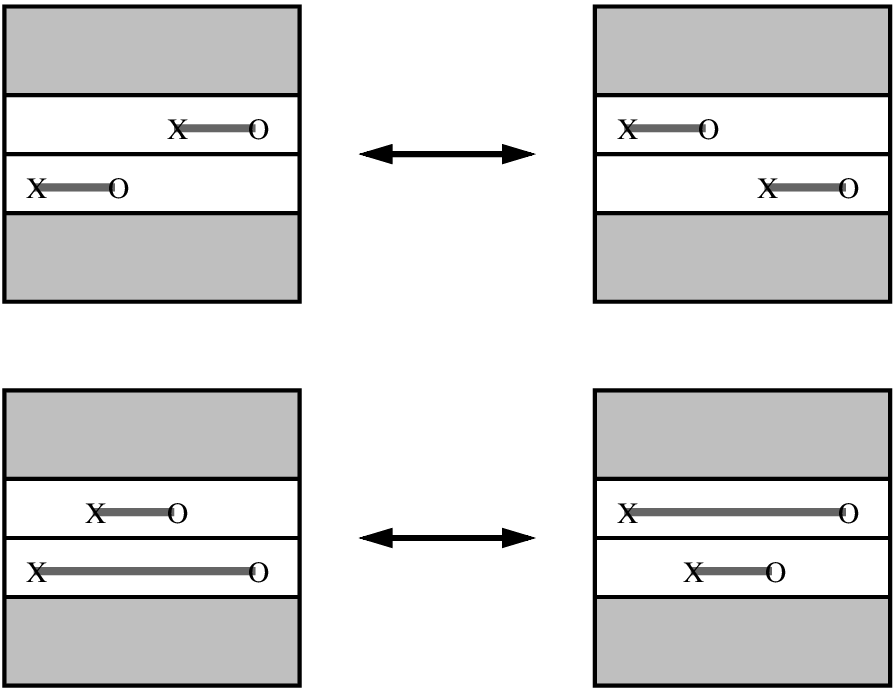}
\caption{{\bf Commutation.} Adjacent rows (or columns) can be interchanged when the markings are situated as above.} \label{commutation}
\end{figure}
\end{enumerate}

\medskip

In \cite{Cromwell}, Cromwell presents a similar set of moves for grid diagrams, except the second commutation move in Figure~\ref{commutation} is replaced with the following move instead:

\smallskip

\begin{itemize}
\item {\bf Cyclic Permutation.} The rows or columns of a grid diagram can be cyclically permuted without changing the link type.

\medskip

\begin{figure}[H]
\includegraphics[scale=.6]{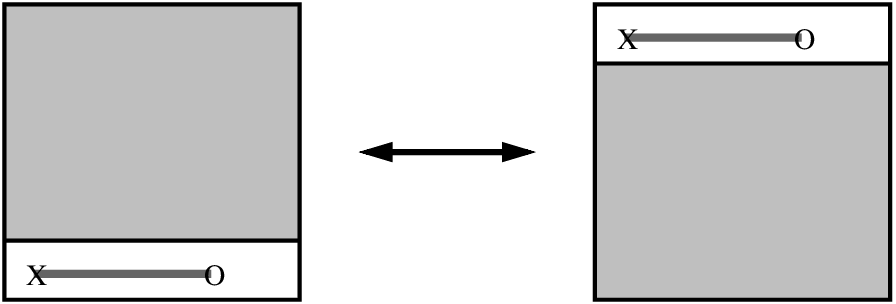}
\caption{{\bf Cyclic Permutation.} The diagram on the left is transformed into the diagram on the right by moving the bottom row to the top of the diagram.} \label{cyclicperm}
\end{figure}
\end{itemize}

\medskip

The second commutation move is equivalent to the cyclic permutation move.

\begin{proposition}
Let $G$ and $G^\prime$ be two grid diagrams representing the same link. Then $G$ can be transformed into $G^\prime$ via a sequence of stabilization moves and commutation moves. \label{gridmovesprop}
\end{proposition}
\begin{proof}
It is enough to show that a cyclic permutation move can be accomplished through a sequence of stabilization and commutation moves. Figure \ref{gridseq} shows how to accomplish such a cyclic permutation move.
\begin{figure}[H]
\includegraphics[scale=.2]{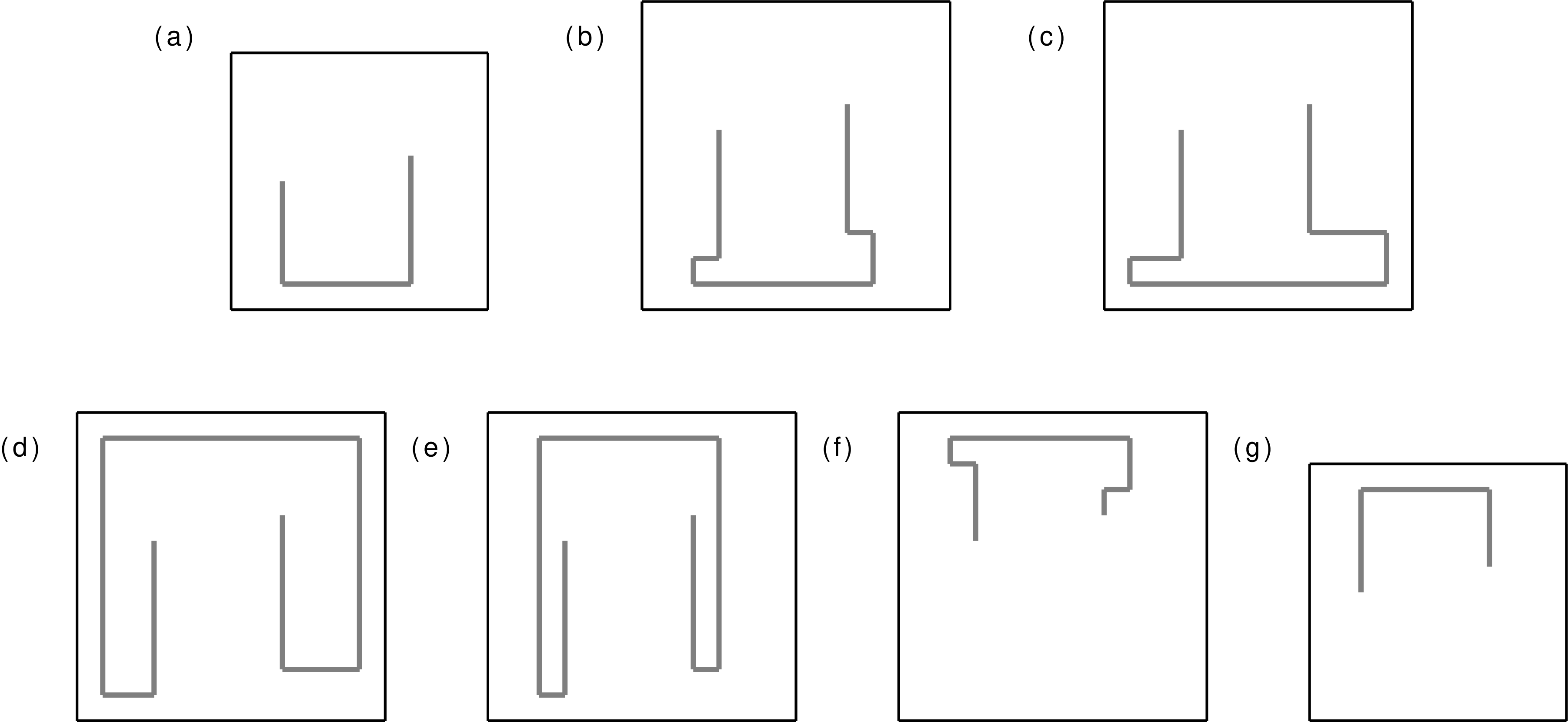}
\caption{Diagram (b) is obtained by stabilizing diagram (a) twice. Diagram (c) is obtained by commuting the vertical segments created in the stabilizations until they are on the leftmost and rightmost column. Diagram (d) is obtained by applying commutation moves repeatedly until the bottommost segment becomes the topmost segment. Diagram (e) is obtained by commuting the leftmost and rightmost columns until they are in the desired place. Diagram (f) is obtained by moving the two short horizontal lines on the bottom until they are all the way to the top. Finally, diagram (g) is obtained by destabilizing twice.} \label{gridseq}
\end{figure}
\end{proof}

\medskip

\subsection{Invariant Cube Moves}
There are corresponding stabilization and commutation moves for a cube diagram. To state these moves, we develop some terminology. Let $x_1$ and $x_2$ be two lines segments in $\mathbb{R}^3$ that are parallel to the $x$-axis. If the endpoints of the projections of $x_1$ and $x_2$ to the $x$-axis are distinct and alternate between the segments, then $x_1$ and $x_2$ are said to {\it interleave}. There is an analogous version of interleaving for pairs of segments parallel to the $y$-axis or $z$-axis. Let $F_1$ and $F_2$ be two $x$-flats in a cube diagram representing some link $L$. Suppose $y_i$ is the segment of $L$ in $F_i$ that is parallel to the $y$-axis and $z_i$ is the segment of $L$ in $F_i$ that is parallel to the $z$-axis, for $i=1$ and $2$. Then $F_1$ and $F_2$ are said to {\it interleave} if either $y_1$ and $y_2$ interleave or $z_1$ and $z_2$ interleave.

\medskip

The following moves do not change the isotopy type of the link embedded in the cube diagram.

\medskip

\begin{enumerate}
\item {\bf Cube Stabilization.} A cube stabilization move increases the cube number of a diagram by one. Suppose $\Gamma$ is a cube diagram with grid number $n$ and label the markings of $\Gamma$ by $\{X_j\}_{j=2}^{n+1},$ $\{Y_j\}_{j=2}^{n+1}$ and $\{Z_j\}_{j=2}^{n+1}$. Let $Z_{i-1},X_i,Y_i$ and $Z_i$ be consecutive vertices in the link described by $\Gamma$.   To stabilize at a marking $X_i$, insert a new $x$-flat into $\Gamma$ adjacent to the $x$-flat containing $X_i$ so that the new flat is between the $x$-flats containing $X_i$ and $Z_i$. Insert a new $y$-flat into $\Gamma$ adjacent to the $y$-flat containing $Y_i$ such that the new $y$-flat is not between the $y$-flat containing $Z_{i-1}$ and the $y$-flat containing $Y_i$. Move $Y_i$ and $Z_i$ into the new $y$-flat by increasing or decreasing its $y$-coordinate by one. Insert a new $z$-flat between the $z$-flats containing $X_i$ and $Y_i$ such that the new flat is adjacent to $X_i$. There is a unique way to place markings $X_1$, $Y_1$ and $Z_1$ into the new $z$-flat to get a new cube diagram $\Gamma'$.   The operation of going between $\Gamma$ to $\Gamma'$ is called a {\em cube stabilization move} (the inverse of the construction above is often called destabilization).  To define stabilizations at $Y_i$ (respectively $Z_i$), follow the same procedure above after cyclically permuting the coordinates and markings once (respectively twice).  See Figure~\ref{cubestab}.

\begin{figure}[H]
\includegraphics[scale=.3]{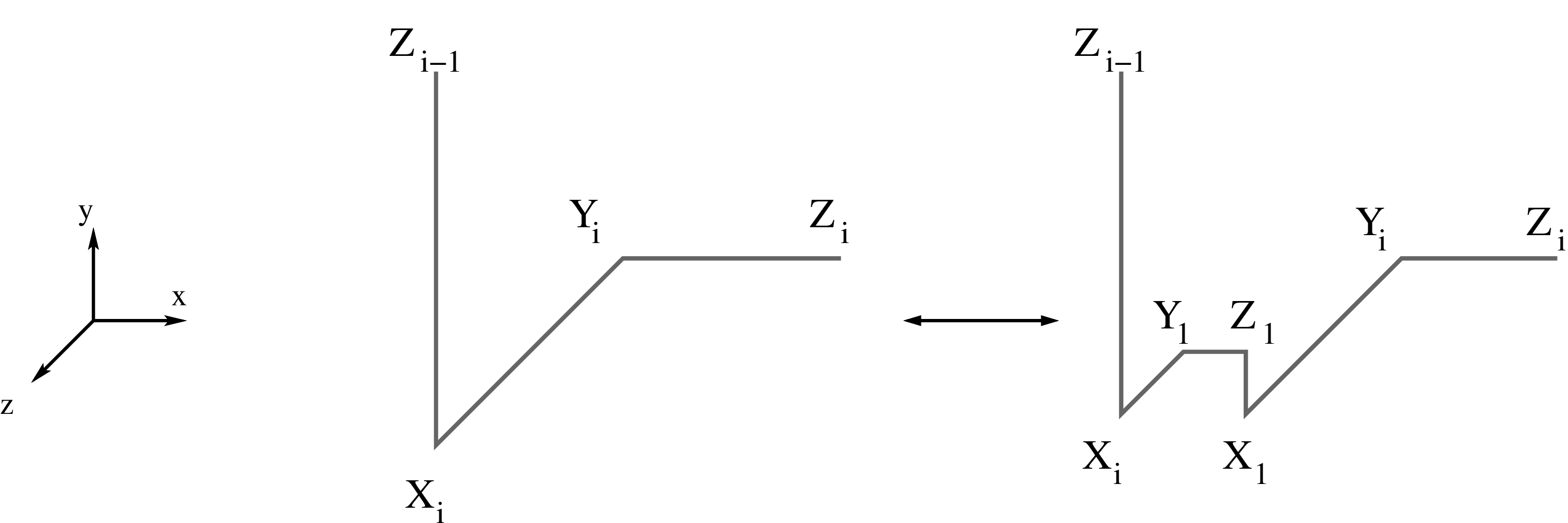}
\caption{{\bf Cube Stabilization.} An example of a stabilization at $X_i$.} \label{cubestab}
\end{figure}

\medskip

\item {\bf Cube Commutation.} Let $F_1$ and $F_2$ be two adjacent flats. Suppose $F_1$ and $F_2$ are not interleaved. Moreover, suppose that interchanging $F_1$ with $F_2$ results in a cube diagram (i.e. the crossings in each projection satisfy the crossing conditions). Then interchanging the flats $F_1$ and $F_2$ is a {\em cube commutation move}.  Up to rotations of the cube diagram that cyclically permute the coordinates, all cube commutation moves can be generated by the four commutation moves listed in Figure~ \ref{cubecom} together with cube stabilizations moves.

\begin{figure}[H]
\includegraphics[scale=.25]{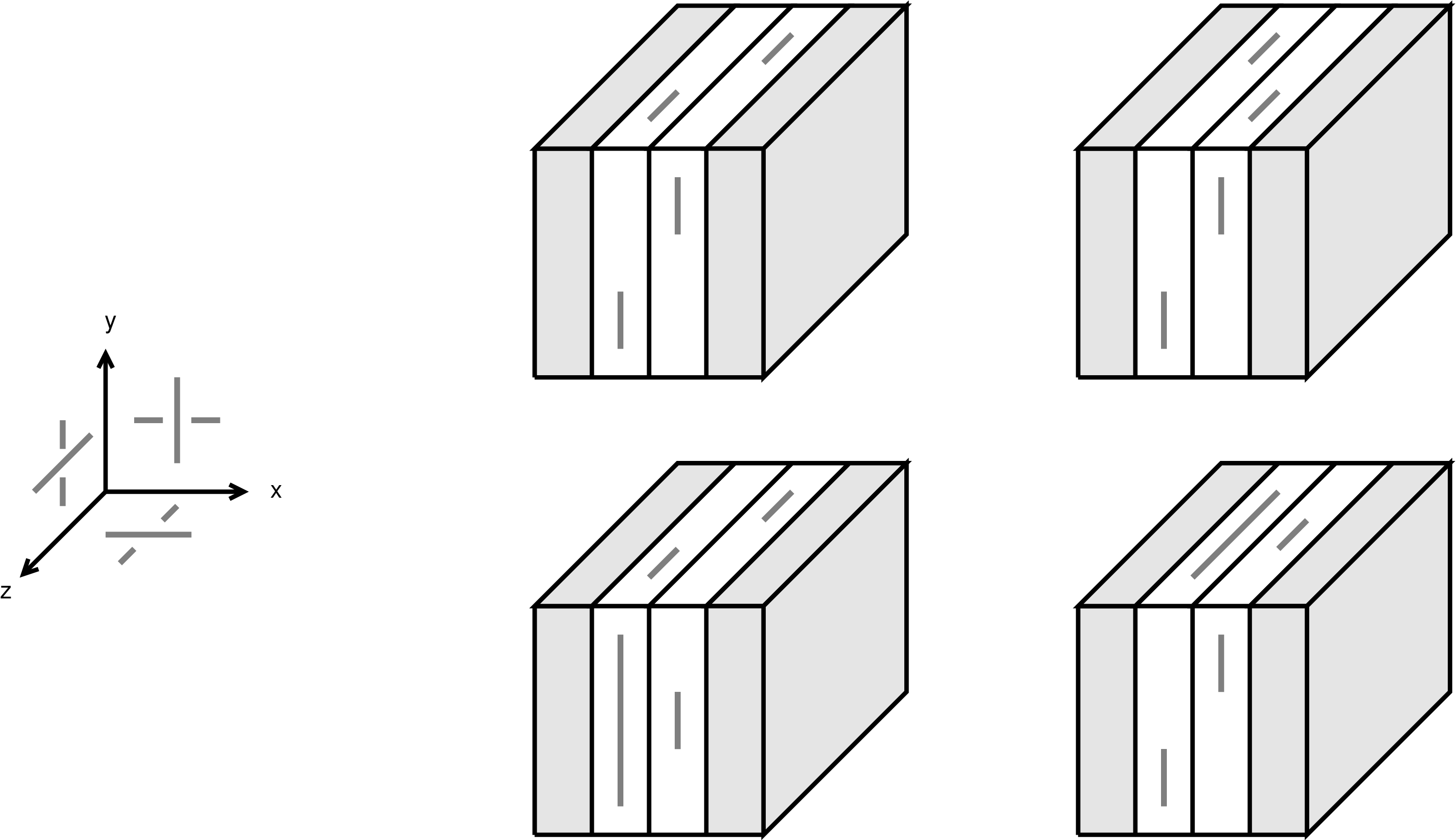}
\caption{{\bf Cube Commutation.} In each of the four schematics above, an interchange of the $x$-flats with indicated grid projections (or vice versa)  is a cube commutation.   The schematics are exactly the same for $y$-flats and $z$-flats.  These four cube commutation moves together with the cube stabilization move generate all cube commutation moves.} \label{cubecom}
\end{figure}
\end{enumerate}

\medskip

Notice that a crossing  can be introduced in a projection by following a cube stabilization with cube commutation(s).  In this way it is possible to introduce a Reidemeister~I move in one or more of the three grid projections (see Figure~\ref{cube-stab-move}).

\medskip

\mbox{} \begin{figure}[H]
\includegraphics[scale=.25]{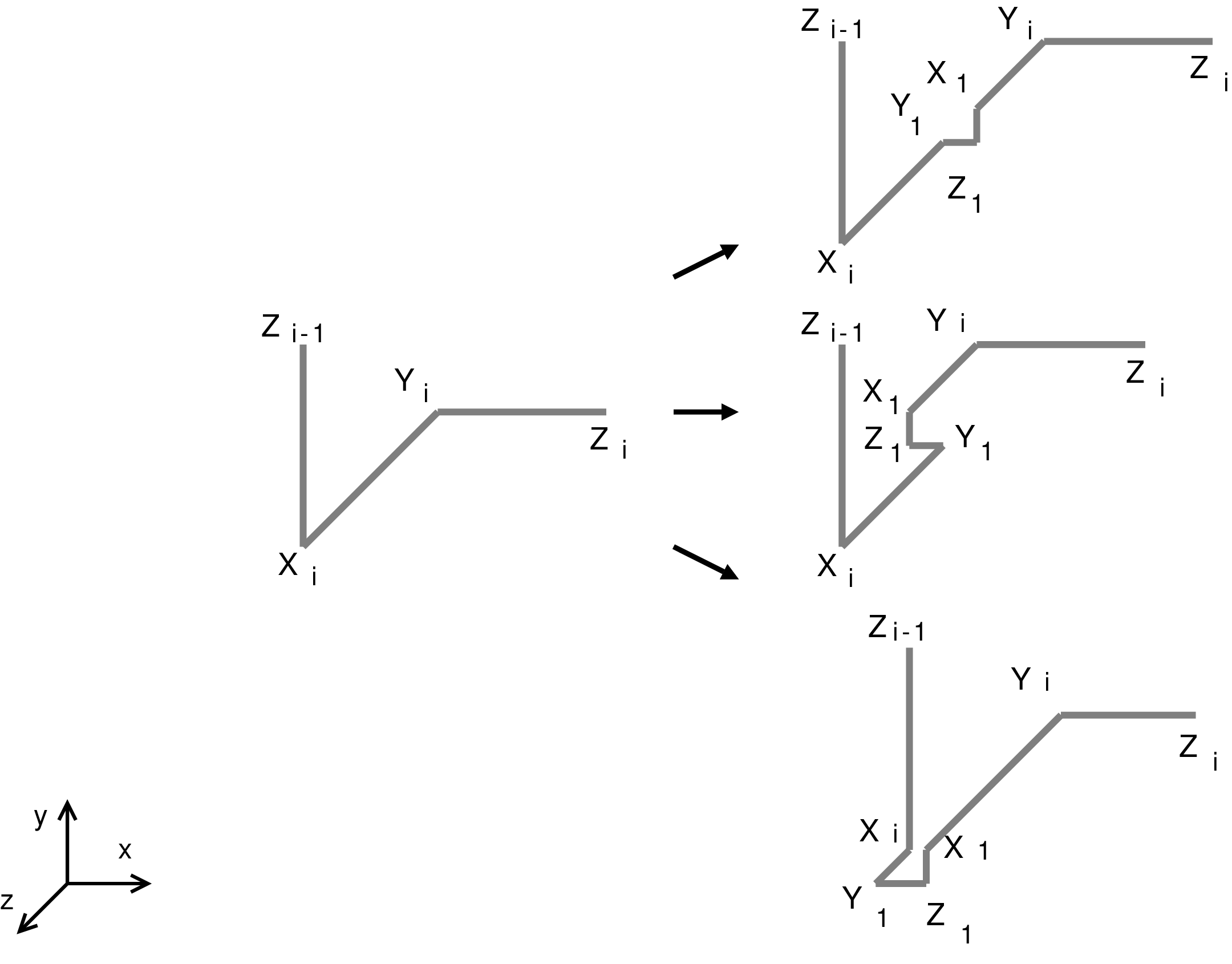}
\caption{A  grid stabilization move in one of the projections of a cube can be realized by the projection of a cube stabilization.  Reidemeister~I moves in any of the projections can be achieved by a cube stabilization followed by cube commutation move(s).} \label{cube-stab-move}
\end{figure}

\medskip

The projection of a cube stabilization move to any face is a stabilization move on that grid diagram.  The relationship between the cube and grid versions of the commutation move is slightly more nuanced.  Let $G$ be the grid diagram that is the $(z,x)$-projection of the cube diagram $\Gamma$. Suppose there is a commutation move on $G$ that corresponds to interchanging two $x$-flats, $F_1$ and $F_2$. Interchanging $F_1$ and $F_2$ may fail to be a cube move in two ways (see Figure~\ref{badcom2} for examples):

\begin{enumerate}
\item Interchanging $F_1$ and $F_2$ may break the crossing conditions.\\

\item Interchanging $F_1$ and $F_2$ may not be an isotopy of the link at all.\\
\end{enumerate}

\medskip

\begin{figure}[H]
\includegraphics[scale=.35]{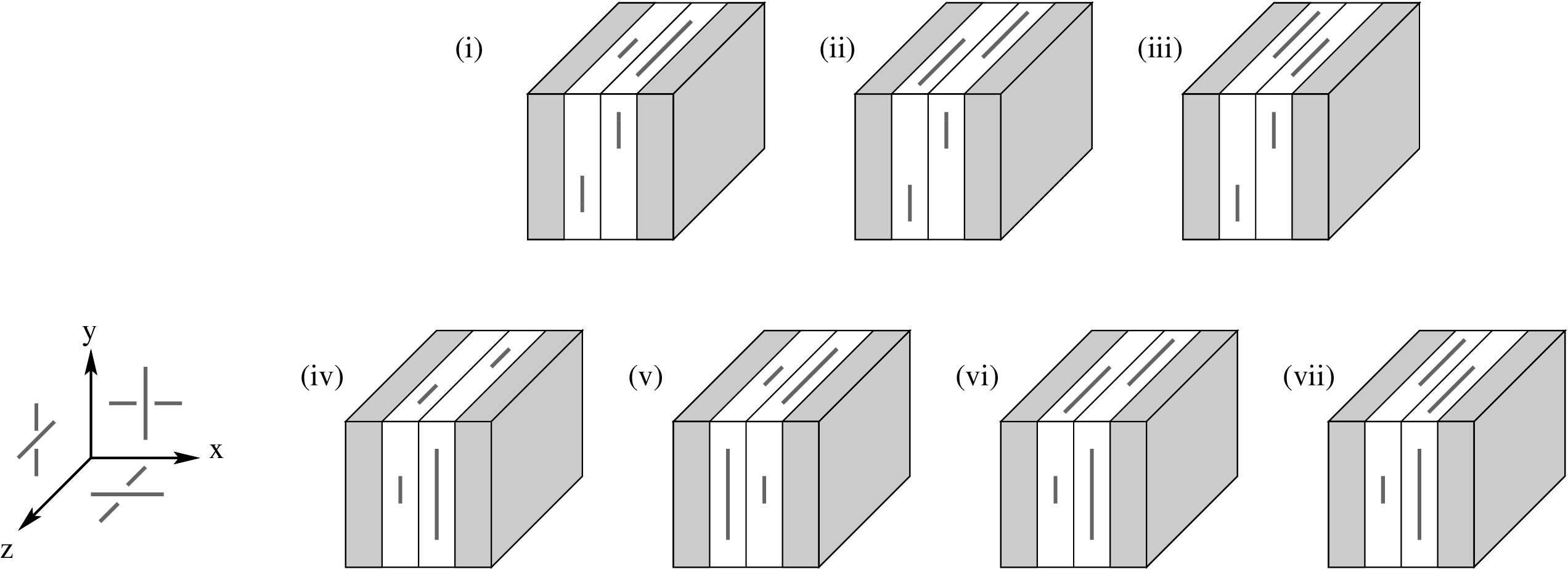}
\caption{These diagrams illustrate cases where exchanging the indicated adjacent $x$-flats fail to be a cube commutation move.} \label{badcom2}
\end{figure}

\medskip

\subsection{Equivalent Cube Diagrams}
Our goal is to show that if two cube diagrams represent the same oriented link, then one can be transformed into the other via a sequence of cube stabilization and cube commutation moves. 

Let $L$ be a piecewise linear link, and let $r$ be a rectangle such that the intersection of $L$ and $r$ (including its interior) is a single piecewise linear arc $a$ that is either $1$, $2$, or $3$ sides of $r$. Additionally, assume that $r$ is parallel to one of the three coordinate planes. A {\em rectangle move on $L$ through $r$} is performed by removing the piecewise linear arc $a$ from $L$ and replacing it with the complementary sides of $r$. If we say that $\widehat{L}$ is obtained from $L$ via a rectangle move through $r$, then it assumed that $r$ and $L$ meet the conditions necessary to perform a rectangle move. A rectangle move where the piecewise linear arc $a$ consists of $k$ line segments for $k=1,2$, or $3$ is called a $(k,4-k)$-rectangle move. Note that if $L$ is represented by some cube diagram, then the result of performing a rectangle move on $L$ is not necessarily represented by a cube diagram.

The first goal of this section is to show that one can approximate a rectangle move in the setting of cube diagrams by using cube moves. In order to formalize what we mean by approximate, we introduce the notion of strong equivalence. Two links diagrams in the plane are {\em strongly equivalent} if one can be transformed into the other without using any Reidemeister moves (i.e., only planar isotopies). 

The main tool used in proving that any two cube diagrams representing isotopic links can be connected with a sequence of cube moves is a lemma that says if a cube diagram $\Gamma$ describes a link $L$ and $L$ is related to a link $\widehat{L}$ via a rectangle move, then there exists a cube diagram $\Gamma^\prime$ representing a link $L^\prime$ such that the corresponding 3 coordinate plane link projections of $\widehat{L}$ and $L^\prime$ are strongly equivalent. 

In Lemma \ref{lemma:Reidemeister}, we describe several local moves on cube diagrams that will be useful in proving our main lemma. 

\begin{lemma}
\label{lemma:Reidemeister}
Let $\Gamma$ be a cube diagram representing the link $L$.  Let $\widehat{L}$ be the piecewise linear link obtained by interchanging two adjacent flats of $\Gamma$.  Suppose that $L$ and $\widehat{L}$ are identical in one link projection and strongly equivalent in another. Suppose that in the remaining link projection, $L$ and $\widehat{L}$ are related by either a Reidemeister II or Reidemeister III move. Then there exists a cube diagram $\Gamma^\prime$ representing the link $L^\prime$ such that $\Gamma^\prime$ can be obtained from $\Gamma$ via a sequence of cube moves.  Furthermore, all three of the corresponding coordinate plane link projections of $\widehat{L}$ and $L^\prime$ are strongly equivalent.
\end{lemma}

\begin{proof}
Let $\widehat{\Gamma} = (C, \{\mathcal{X},\mathcal{Y},\mathcal{Z}\}, \widehat{L})$ be a data structure obtained by interchanging the two flats of $\Gamma$.   If $\widehat{\Gamma}$ is a cube diagram, then the result is obvious and one can take $\Gamma^\prime = \widehat{\Gamma}$. Otherwise $\widehat{\Gamma}$ satisfies the marking conditions but fails the crossing conditions in one coordinate plane link projection.   In that projection, a Reidemeister II or III move is performed. Without loss of generality, suppose that the $(y,z)$-projection is the projection where $\widehat{\Gamma}$ does not satisfy the crossing conditions. 

\medskip

Assume that the $(y, z)$-projections of $\Gamma$ and $\widehat{\Gamma}$ are related by a Reidemeister II move.  Then the $(y, z)$-projection of $\widehat{\Gamma}$ contains either  $1$ or $2$ crossings that do not satisfy the crossing conditions (cf. Figure \ref{lemmatrick1pic1}).
\begin{figure}[h]
\includegraphics[scale=.6]{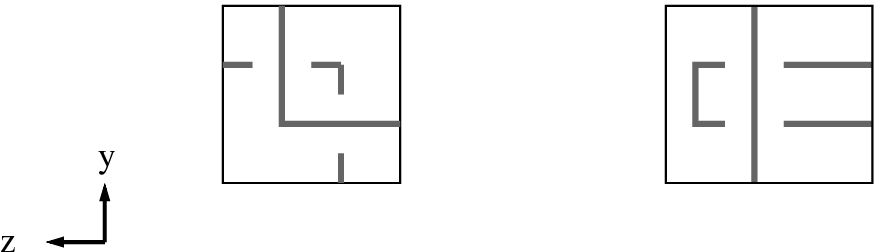}
\caption{A local neighborhood of a projection of $\widehat{\Gamma}$. The diagram on the left has one crossing that does not satisfy the crossing conditions and the diagram on the right has two.}\label{lemmatrick1pic1}
\end{figure}

Figures \ref{lemmatrick1pic2} and \ref{lemmatrick1pic3} depict a sequence of local diagrams starting with the $(y,z)$-projection of $\Gamma$ and ending in two diagrams that are strongly equivalent to the two diagrams of Figure \ref{lemmatrick1pic1} respectively.  Each of the grid diagrams in the sequence is the projection of a cube diagram and adjacent cube diagrams in the sequence are related by a sequence of cube moves.  The $(x, y)$-projection and $(z, x)$-projection of each cube diagram depicted in Figures \ref{lemmatrick1pic2} and \ref{lemmatrick1pic3} are strongly equivalent to the corresponding projection of $\widehat{\Gamma}$.
\begin{figure}[h]
\includegraphics[scale=.4]{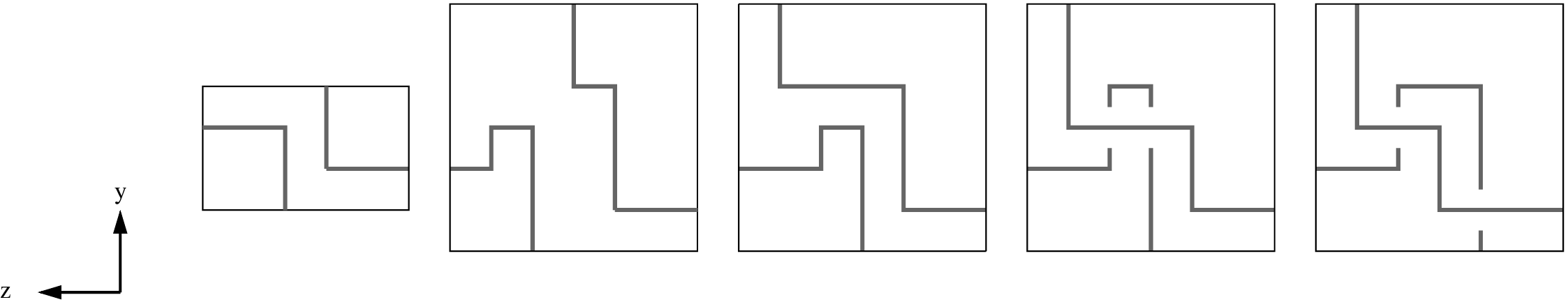}
\caption{The $(y,z)$-projections of the sequence of cube diagrams from $\Gamma$ to $\Gamma^\prime$ for the first diagram of Figure \ref{lemmatrick1pic1}.}
\label{lemmatrick1pic2} 
\end{figure}
\begin{figure}[h]
\includegraphics[scale=.4]{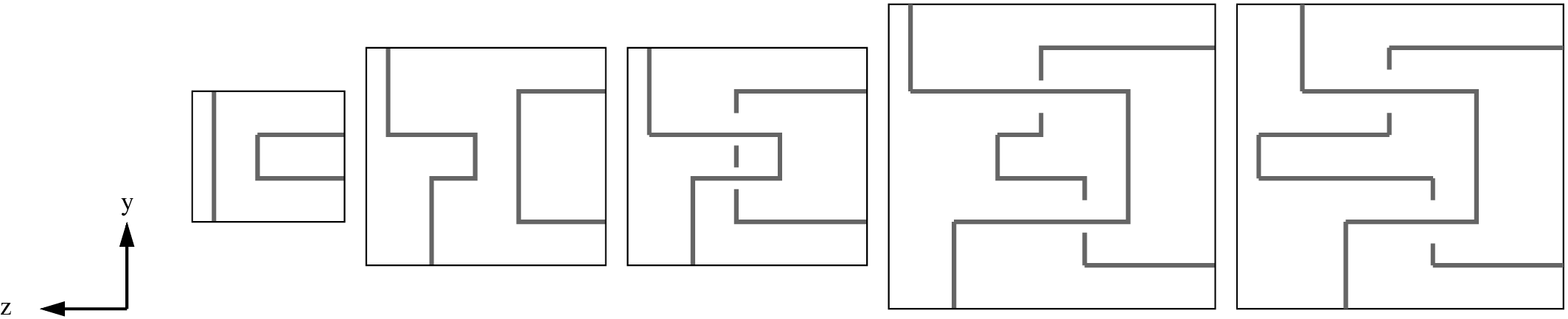}
\caption{The $(y, z)$-projections of the sequence of cube diagrams from $\Gamma$ to $\Gamma^\prime$ for the second diagram of Figure \ref{lemmatrick1pic1}.}
\label{lemmatrick1pic3} 
\end{figure}

Next, assume that the $(y, z)$-projections of $\Gamma$ and $\widehat{\Gamma}$ are related by a Reidemeister III move. Interchanging the two specified flats may break the crossing conditions in the $(y,z)$-projection, as depicted in Figure \ref{badRIII}. However, possibly after performing several cube stabilizations, one can achieve the Reidemeister III move by interchanging two different flats, which is also depicted in Figure \ref{badRIII}.
\begin{figure}[H]
\includegraphics[scale=.25]{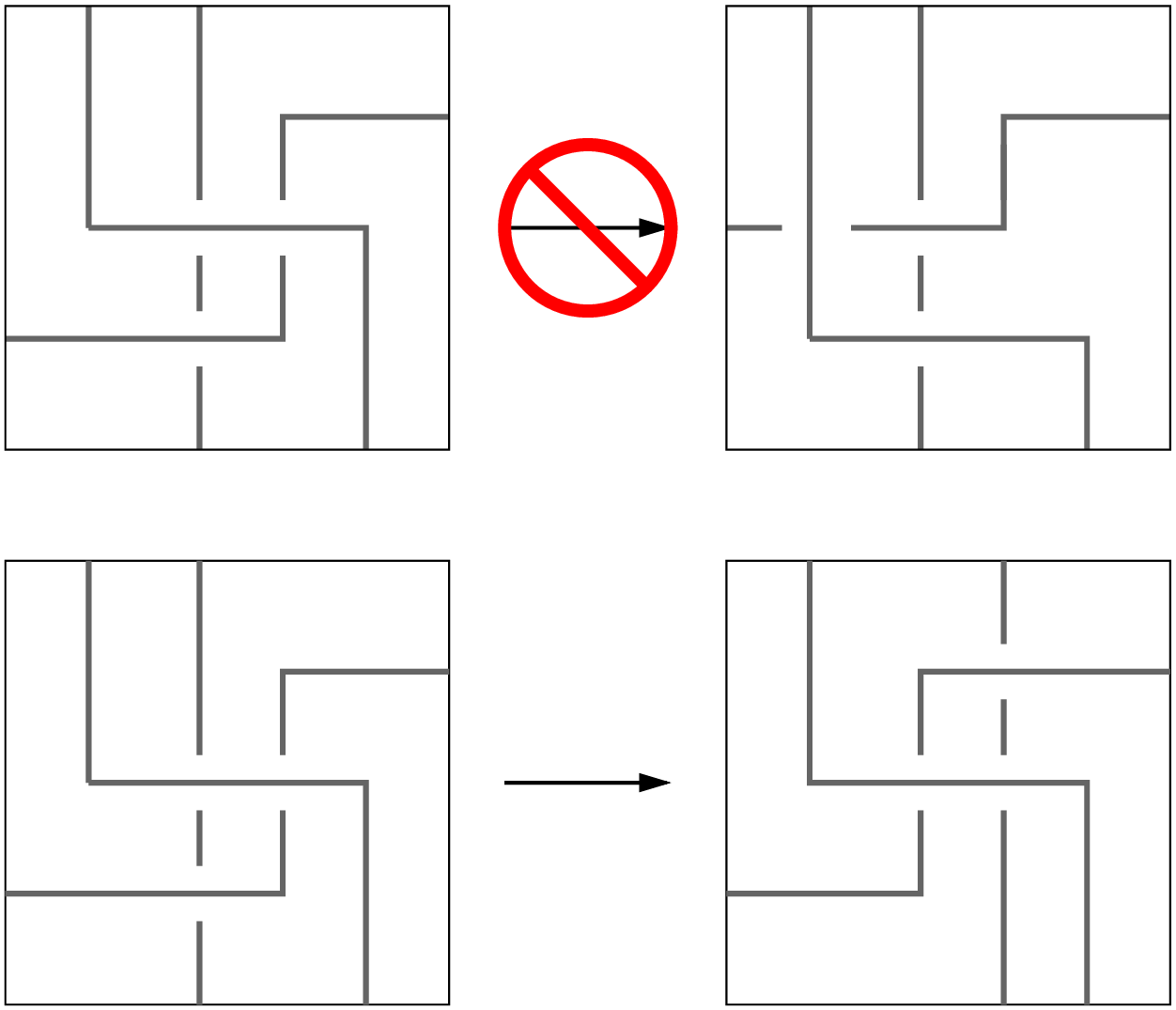}
\caption{Interchanging two adjacent flats may cause the crossing conditions to fail (top diagram). However, two other flats can be interchanged which results in a Reidemeister III move in the $(y,z)$-projection (bottom diagram).}
\label{badRIII}
\end{figure}

Figure \ref{lemmatrick2pic} shows the different ways a Reidemeister III move can be performed in the $(y,z)$-projection. The grid commutation moves of Figure \ref{lemmatrick2pic} may always be lifted to cube commutation moves since if there were any obstruction, one could perform a cube stabilization first, and then move the obstructing piece out of the local picture. These cube commutations preserve strong equivalence of diagrams in all 3 coordinate plane link projections.
\begin{figure}[h]
\includegraphics[scale=.25]{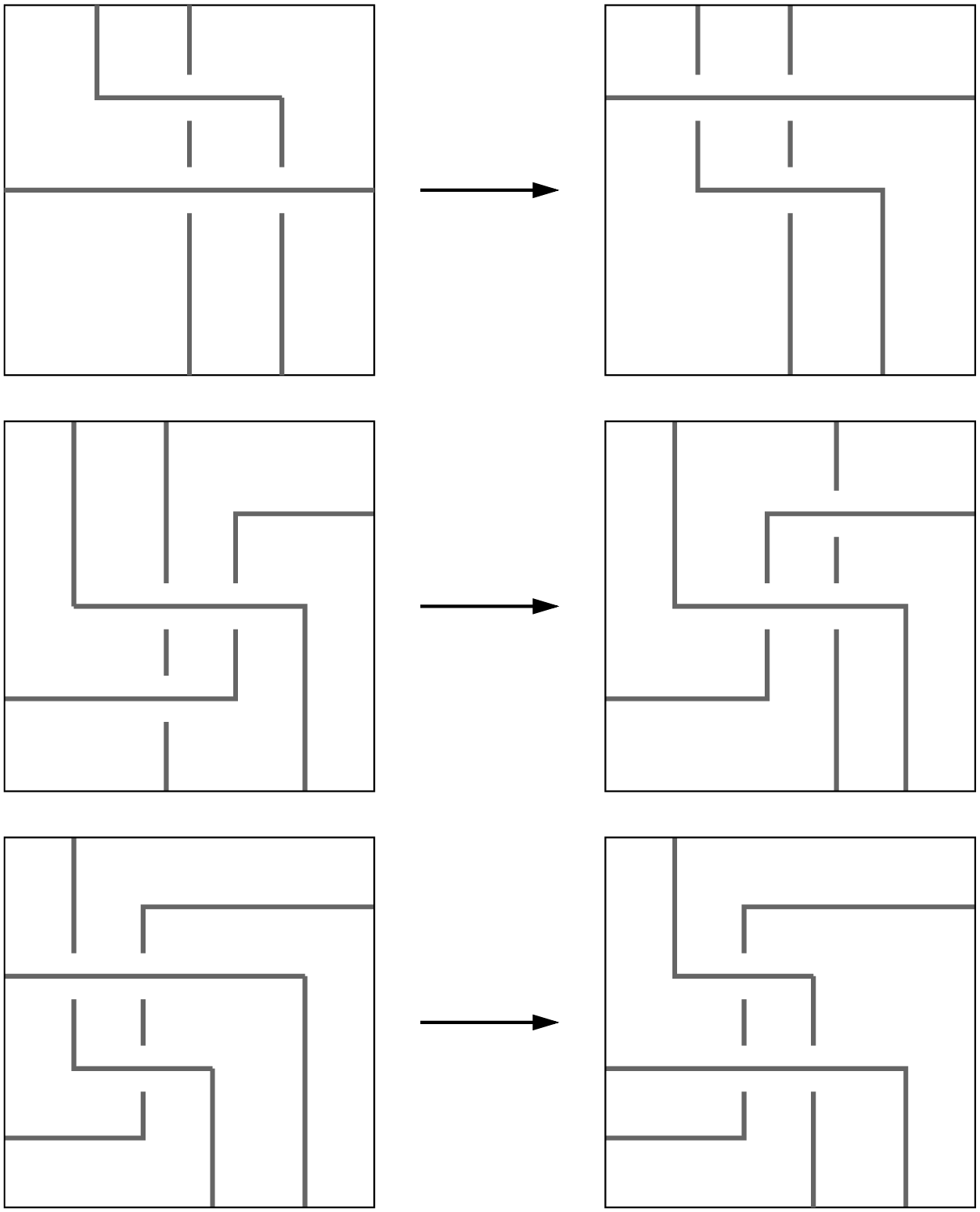}
\caption{Reidemeister III moves in the projection $P$.}
\label{lemmatrick2pic}
\end{figure}

The resulting cube diagram $\Gamma^\prime$ represents a link $L^\prime$ that is strongly equivalent to $\widehat{L}$ in each projection.

\end{proof}

Lemma \ref{lemma:Reidemeister} shows how to approximate isotopies induced by interchanging adjacent flats where the crossing conditions fail in one projection only. However, interchanging two adjacent flats may cause the crossing conditions to fail in more than one projection. If this is the case, then a cube stabilization (as in Figure \ref{separatefig}) will break the isotopy into two flat interchanges, where the crossing conditions in each interchange only fails in one projection.
\begin{figure}[H]
\includegraphics[scale=.3]{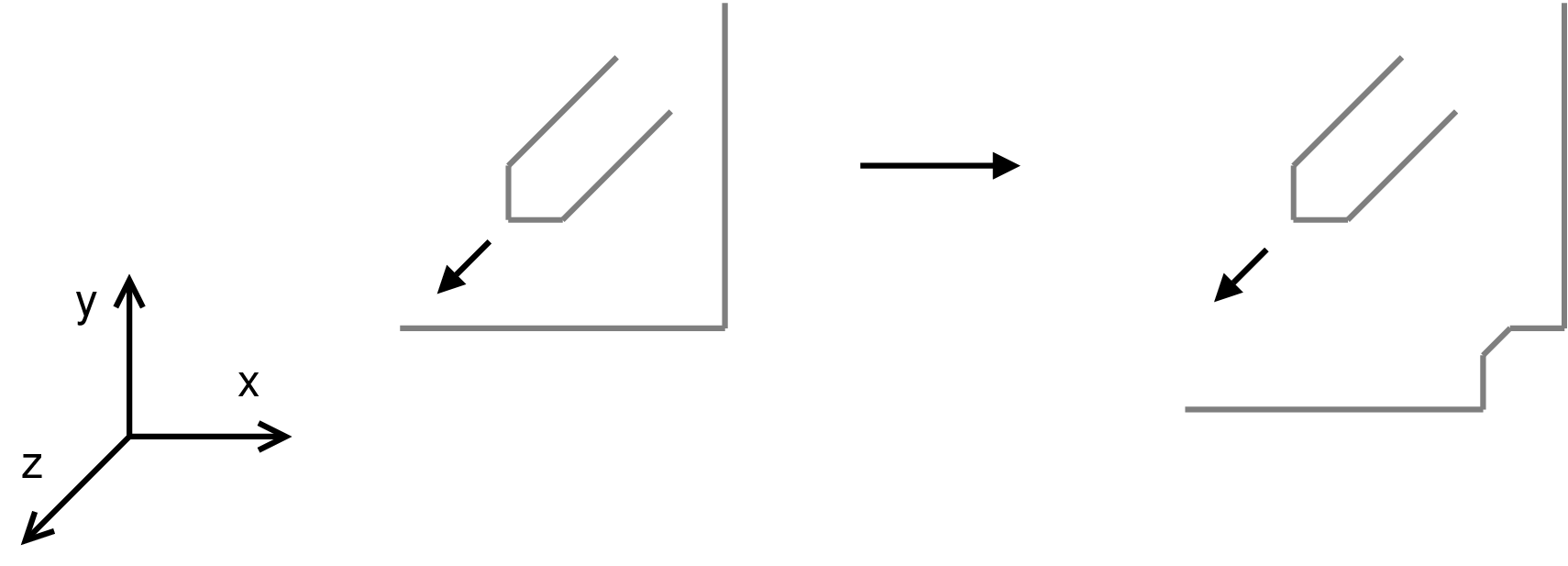}
\caption{Stabilize in order to isolate the failure of crossing conditions to one projection.} \label{separatefig}
\end{figure}

The following lemma is the main tool used to prove that cube diagrams representing isotopic links are related by a sequence of cube moves.

\begin{lemma}[Isotopy Lemma]
Let $\Gamma$ be a cube diagram representing the link $L$, and let $\widehat{L}$ be obtained via a rectangle move on $L$ through $r$. Then there exists a cube diagram $\Gamma^\prime$ representing a link $L^\prime$ such that $\widehat{L}$ and $L^\prime$ are strongly equivalent in each of their corresponding coordinate plane link projections.  Furthermore,  $\Gamma$ can be transformed into $\Gamma^\prime$ via a sequence of cube moves.
\end{lemma}

The proof is essentially given through a sequence of pictures. Let $m$ be a segment of $L$ that is one side of the rectangle $r$. Figure \ref{biglemmafigure} shows an example of one such $\Gamma$ where $m$ is parallel to the $y$-axis and $r$ is parallel to the $yz$-plane. 
\begin{figure}[H]
\includegraphics[scale=.2]{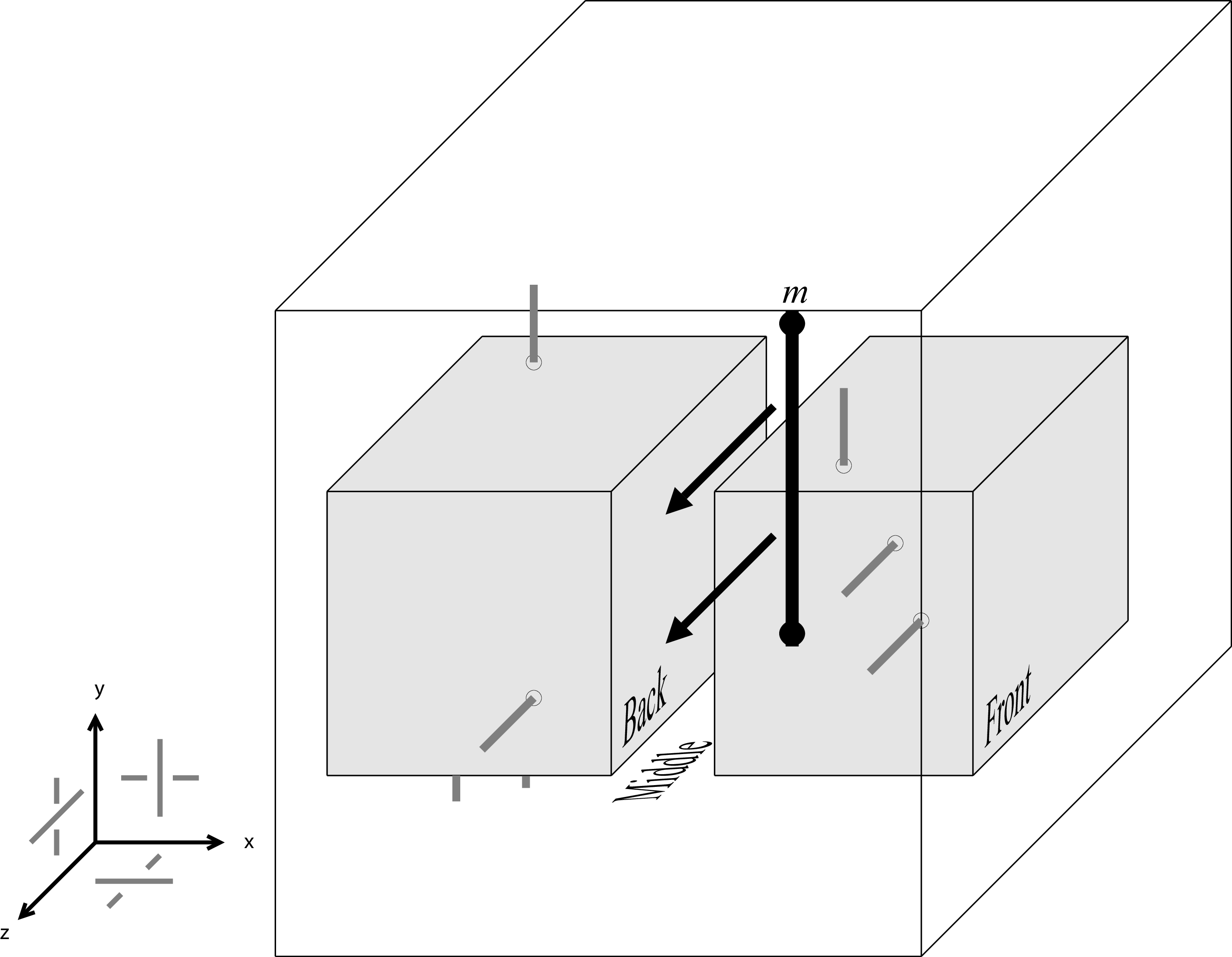}
\caption{An illustration of an isotopy.} \label{biglemmafigure}
\end{figure}

\medskip

The proof splits into two steps.  First, we stabilize $m$ and push a small segment under along the bottom of the rectangle $r$. The second step is to show that it is possible to lift the upper half part of the segment created in the previous step so that it approximates the upper part of the rectangle $r$. This process repeatedly uses the local moves established in Lemma \ref{lemma:Reidemeister}.

\begin{proof}[Proof of the Isotopy Lemma] Since $\Gamma$ is a cube diagram, the rectangle move through $r$ is either a $(1,3)$-rectangle move or a $(2,2)$-rectangle move. Assume that the rectangle move is a $(1,3)$-rectangle move. In the case of a $(2,2)$-rectangle move, one begins with a diagram similar to Figure \ref{isounderfig}, and so the proof for this case is a subproof of the case we present.

Let $m$ be a segment of $L$ that is one side of the rectangle $r$. Assume that $m$ is parallel to the $y$-axis and that the rectangle $r$ is parallel to the $yz$-plane. The argument for this case is easily adapted to all other cases. Also, assume that $m$ is a segment between an $X$-marking and a $Z$-marking, which we call $X_0$ and $Z_0$, and suppose that the $y$-coordinate of $Z_0$ is less than the $y$-coordinate of $X_0$. 

Let $r_{yz}$ be the projection of the rectangle $r$ to the $yz$-plane. In the following argument, various cube stabilizations will be performed to $\Gamma$. Suppose that at each step of the sequence transforming $\Gamma$ to $\Gamma^\prime$ the vertices of $r_{yz}$ have coordinates $(y_m,z_m), (y_m,z_M), (y_M,z_m),$ and $(y_M,z_M)$ for some $y_m<y_M$ and some $z_m<z_M$. If a cube stabilization is performed at some marking whose $y$-coordinate is between $y_m$ and $y_M$, then
the new $y$-coordinates of the vertices of $r_{yz}$ are $y_m$ and $y_M+1$, and if a cube stabilization is performed at some marking whose $z$-coordinate is between $z_m$ and $Z_M$, then the new $z$-coordinates of the vertices of $r_{yz}$ are $z_m$ and $Z_M+1$. Similarly compress $r_{yz}$ if a cube destabilization is performed. In both cases, we will continue to call the new rectangle $r_{yz}$.

One may perform two cube stabilizations along $m$ to create a width 1 by width 1 $z$-cube bend $b_z$. The segment of that $z$-cube bend that is parallel to the $y$-axis is between an $X$-marking and a $Z$-marking, which we call $X_1$ and $Z_1$. The two cube stabilizations also created another new $X$-marking and another new $Z$-marking, which we call $X_2$ and $Z_2$. Figure \ref{firststab} shows an example of the $(x, y)$-projection and $(z, x)$-projection of the cube diagram after the two stabilizations. 
\begin{figure}[H]
\includegraphics[scale=.25]{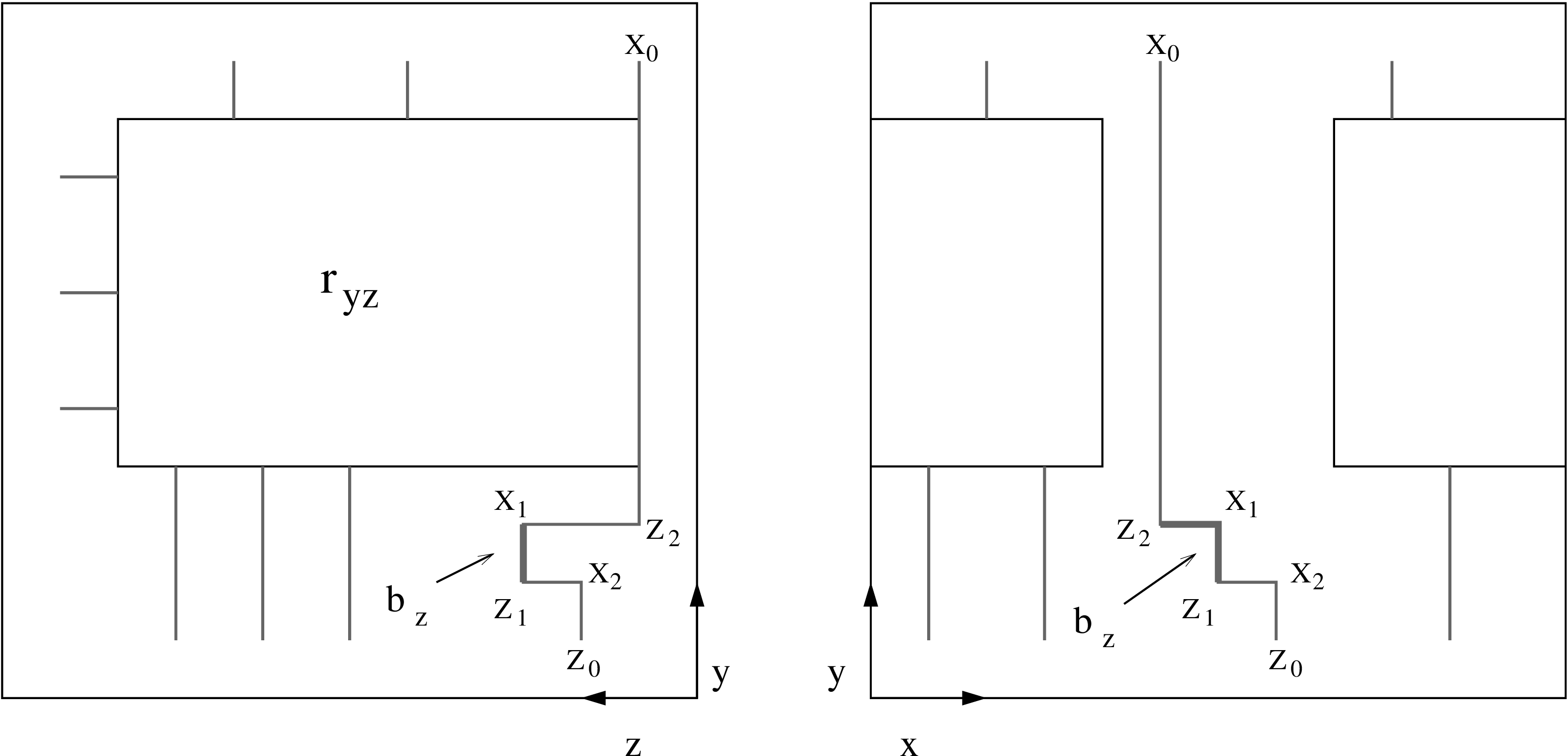}
\caption{Stabilize along the segment $m$ twice to create a $1\times 1$ $z$-cube bend.}
\label{firststab}
\end{figure}

By repeated applications of Lemma \ref{lemma:Reidemeister}, the $z$-flat containing the $z$-cube bend $b_z$ can be commuted until the $z$-coordinate of $X_1$ and $Z_1$ is greater than the largest $z$-coordinate of $r_{yz}$. Note that the crossing conditions in the $(z, x)$-projection must also be satisfied, but this is achieved by applications of Lemma \ref{lemma:Reidemeister} possibly together with cube stabilizations as in Figure \ref{separatefig}. Figure \ref{isounderfig} shows an example of the $(x, y)$-projection and the $(y, z)$-projection of the cube diagram after the commutations.

\begin{figure}[H]
\includegraphics[scale=.2]{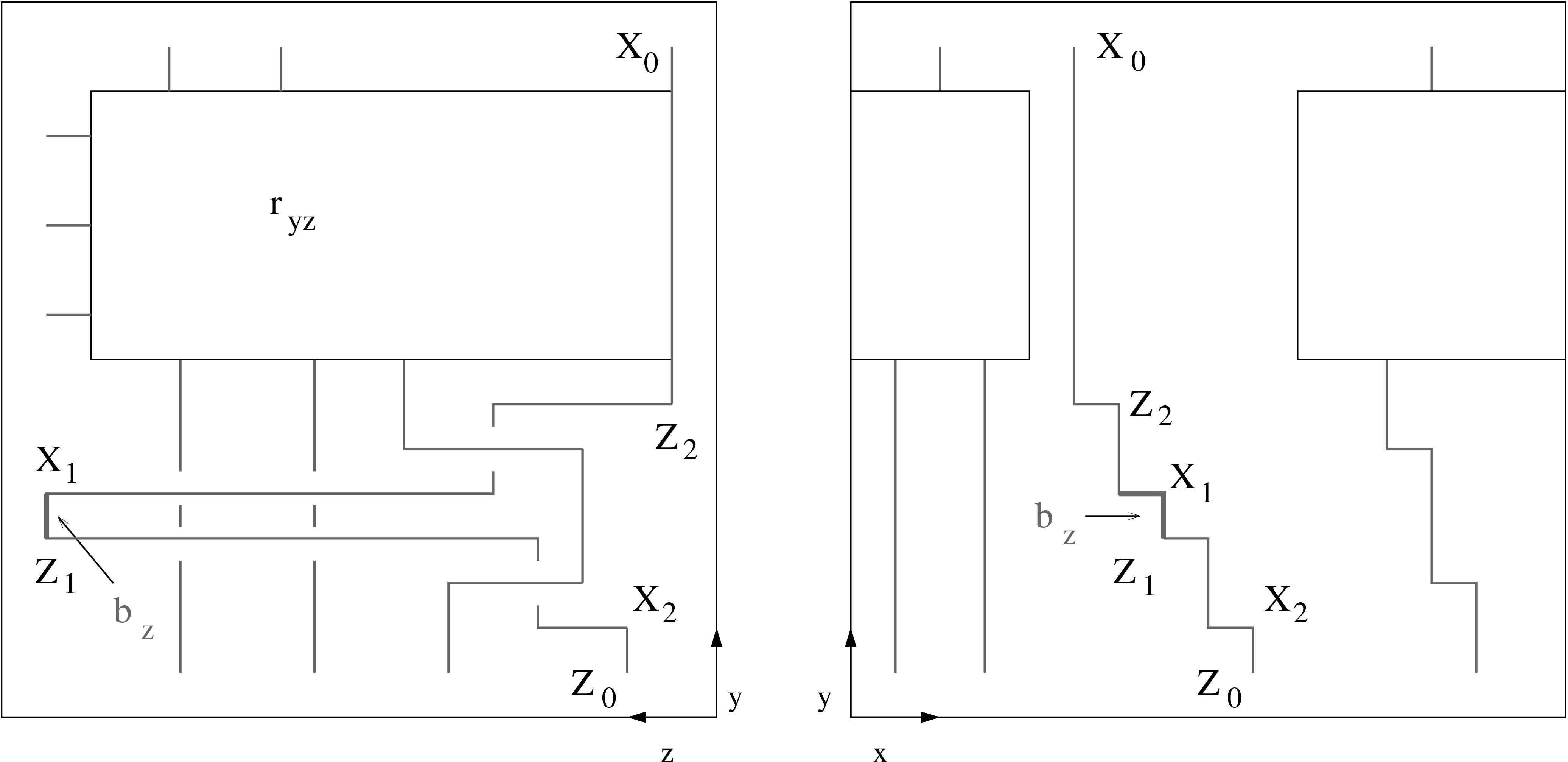}
\caption{Commute the $z$-bend $b_z$ along the $z$-axis.} \label{isounderfig}
\end{figure}

Stabilize the cube diagram twice between $X_1$ and $Z_1$ to form a new $1\times 1$ $y$-cube bend $b_y$.  The segment of $b_y$ which is parallel to the $z$-axis is between an $X$ marking and a $Y$ marking, which we shall name $X_3$ and $Y_3$. The $y$-flat containing $b_y$ can be repeatedly commuted (using Lemma \ref{lemma:Reidemeister}) until the common $y$-coordinate of $X_3$ and $Y_3$ is greater than the greatest $y$-coordinate of $r_{yz}$. Figure \ref{isoalongfig} gives an example of how such a sequence of commutations can be performed. The $(x, y)$-projection of the cube diagram at this stage of the process is strongly equivalent to the $(x, y)$-projection of the original cube diagram $\Gamma$.

\begin{figure}[H]
\includegraphics[scale=.17]{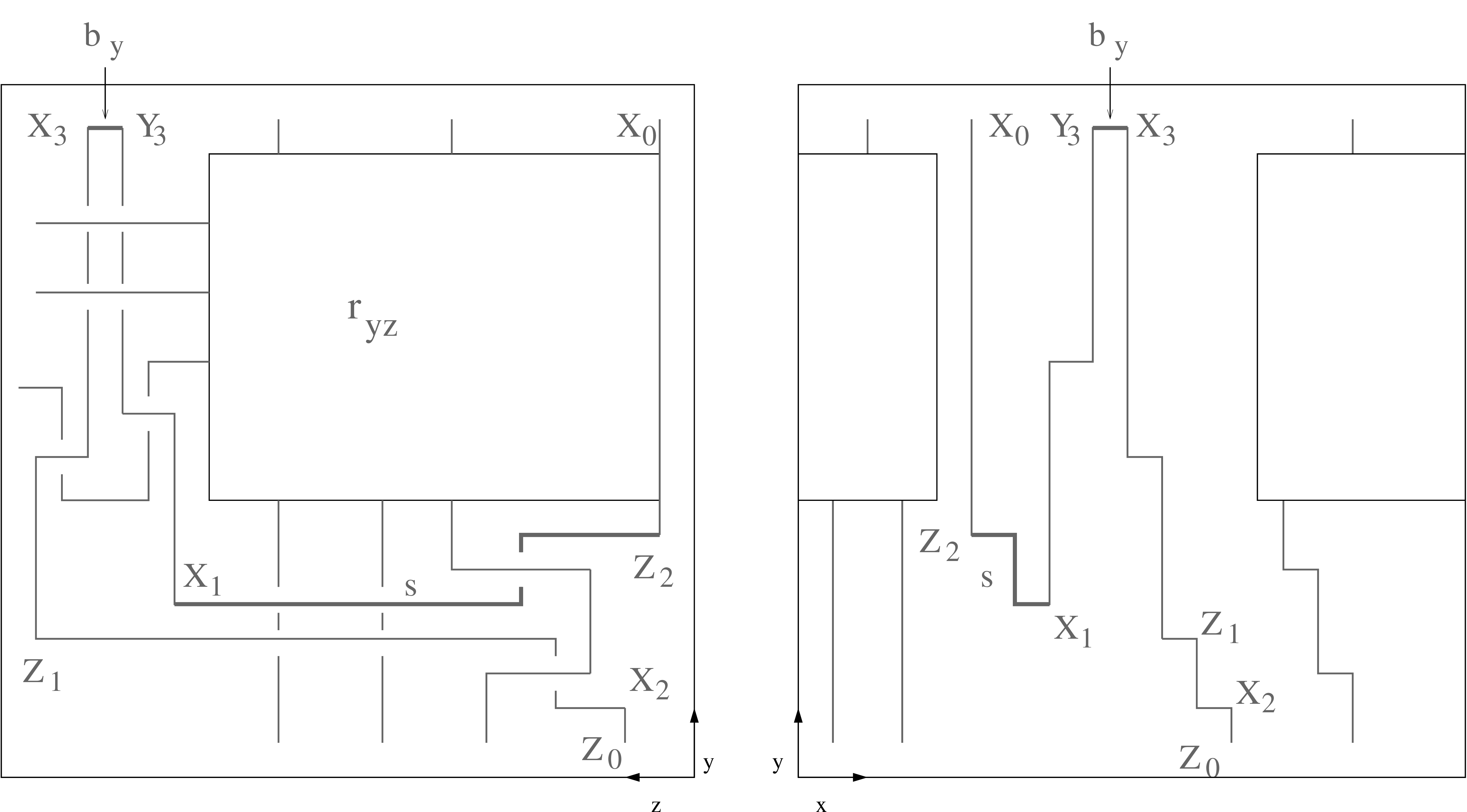}
\caption{Commute the $y$-flat containing $b_y$ along the $y$-axis.} \label{isoalongfig}
\end{figure}

Let $s$ be the collection of line segments from $X_1$ to $Z_2$. The collection of segments $s$ appear in bold in Figure \ref{isoalongfig}. It remains to show that one can pass the segments making up $s$ upward in the $y$-direction so that the $y$-coordinate of all the segments of $s$ is greater than the greatest $y$-coordinate of $r_{yz}$. Figure \ref{endresult} shows the end result of this isotopy.
\begin{figure}[H]
\includegraphics[scale=.15]{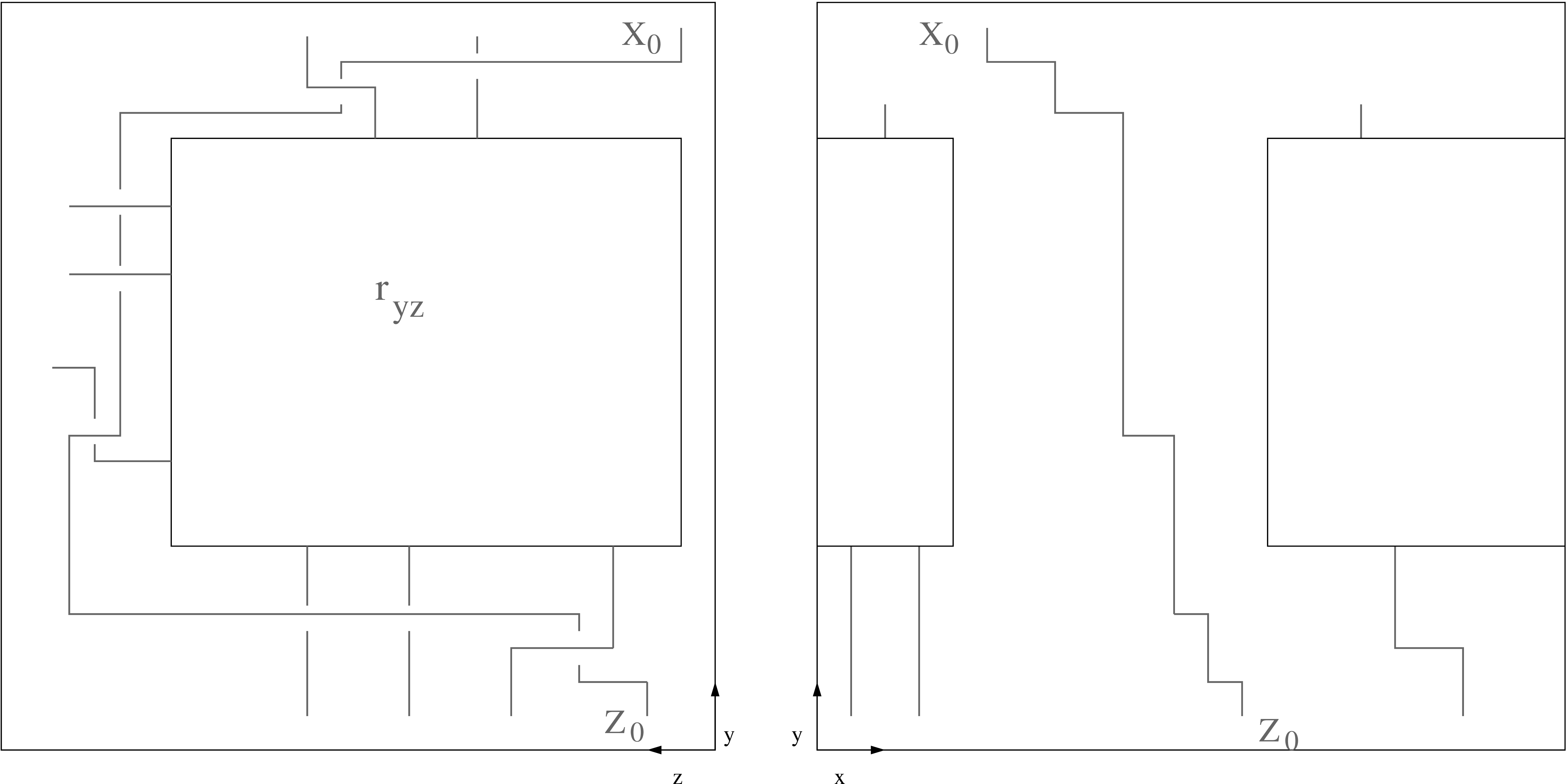}
\caption{The final result of the Isotopy Lemma.}
\label{endresult}
\end{figure}

We perform the isotopy of $s$ in two steps. First, we handle crossings involving the segment $s$ in the $(z,x)$-projection. As the $z$-bend $b_z$ was commuted in the direction of the $z$-axis, it is possible that the link was stabilized in order to ensure that crossing conditions are satisfied in the $(z,x)$-projection. An example of such a crossing is shown in Figure \ref{(z,x)-proj-cross-cond}. Fix a crossing in the $(z,x)$-projection where the arc $s$ passes under another arc of the diagram, and let $t$ be the other segment involved in the crossing. Since $t$ was created as in Lemma \ref{lemma:Reidemeister}, it is only $2$ units long.
\begin{figure}[H]
\includegraphics[scale=.3]{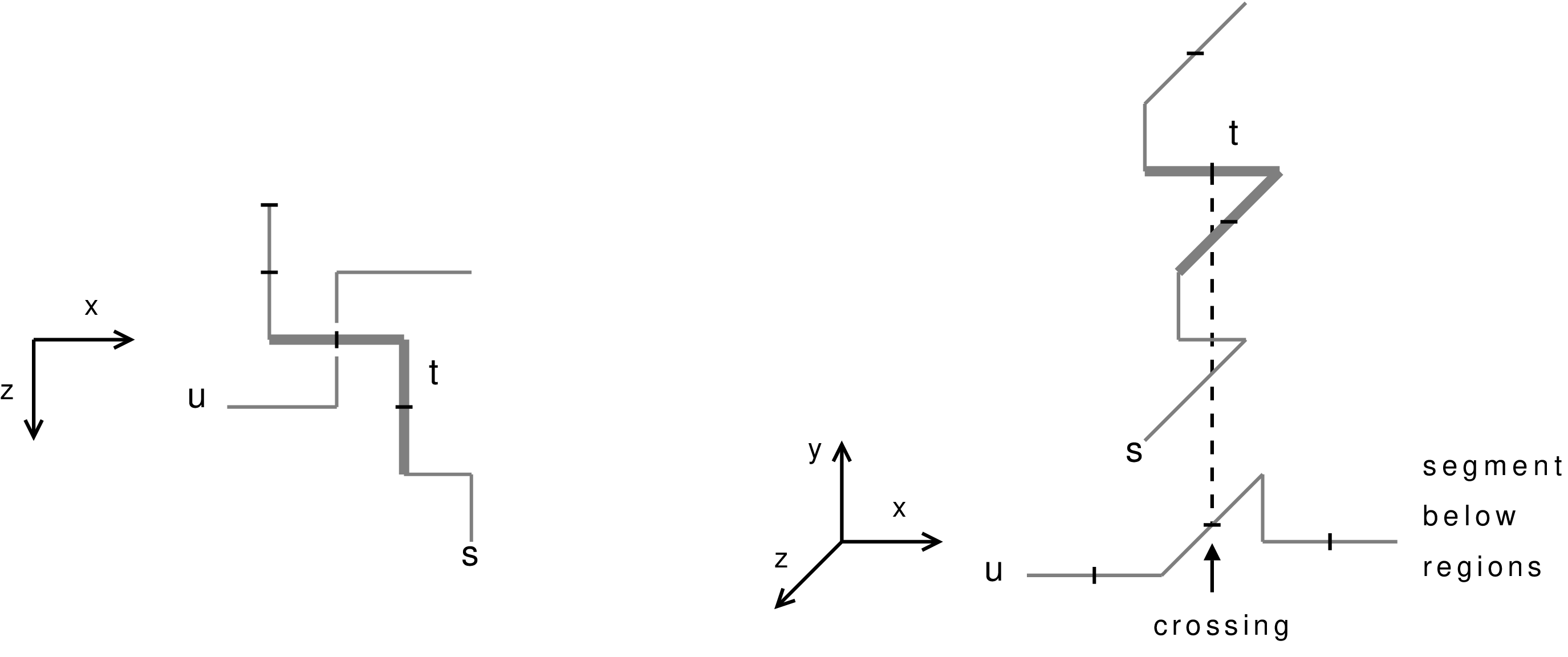}
\caption{Satisfying the crossing conditions in the $(z,x)$-projection.} \label{(z,x)-proj-cross-cond}
\end{figure}

Commute $t$ parallel to the $y$-axis until its $y$ coordinate is greater than the greatest $y$-coordinate of $r_{yz}$. This may require several applications of Lemma \ref{lemma:Reidemeister}. An example is given in Figure \ref{iso-t-fig}. 
\begin{figure}[H]
\includegraphics[scale=.2]{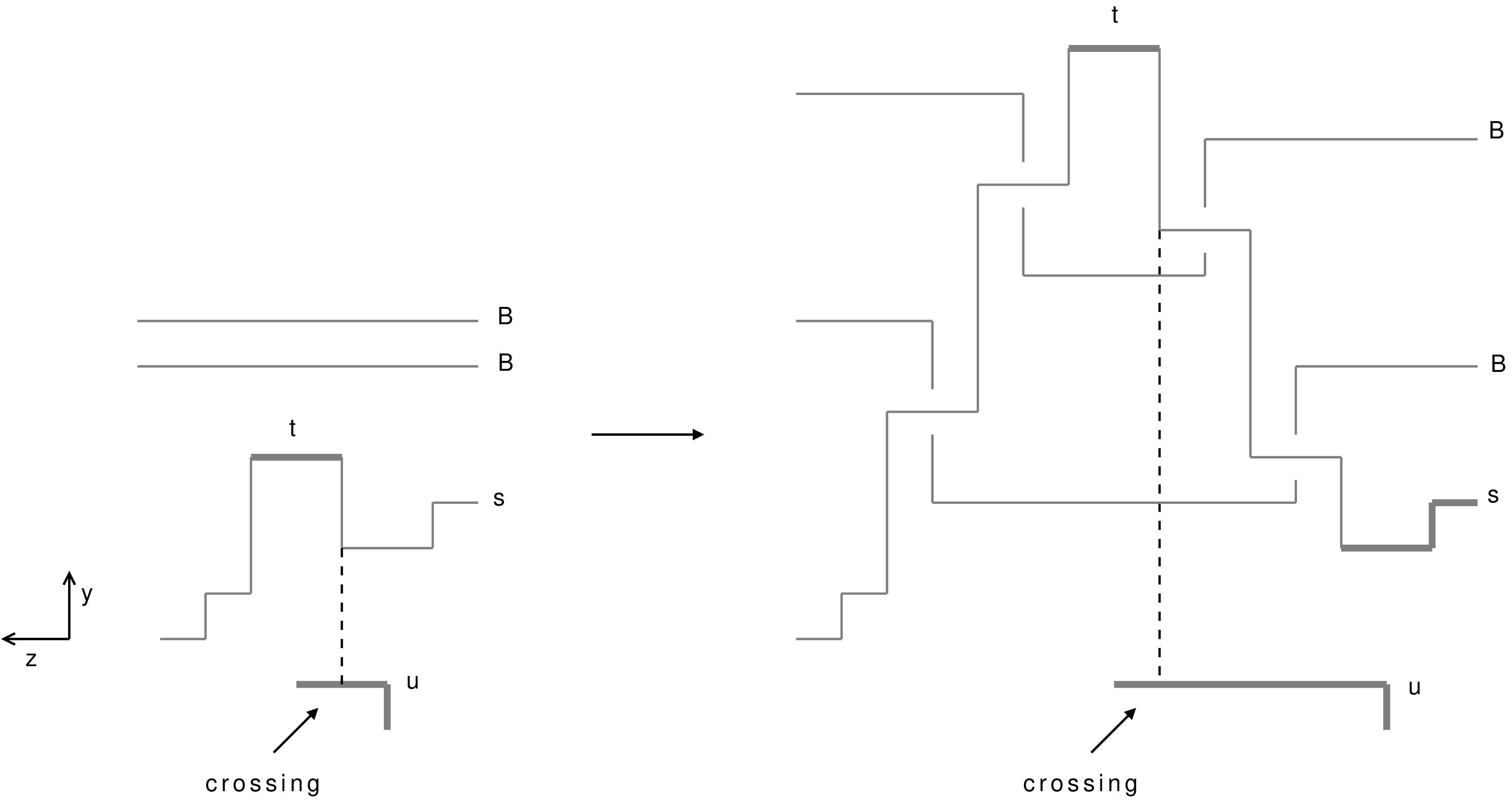}
\caption{Before and after pictures of commuting $t$ to the top of the $r_{yz}$.} \label{iso-t-fig}
\end{figure}

Repeat this process for each crossing in the $(z,x)$-projection where a segment of $s$ passes under another segment in the link. Label the other segments $t_1,\dots t_l$. The $(y,z)$-projection of the cube diagram now appears as in Figure \ref{iso-t-done-fig}.

\begin{figure}[H]
\includegraphics[scale=.15]{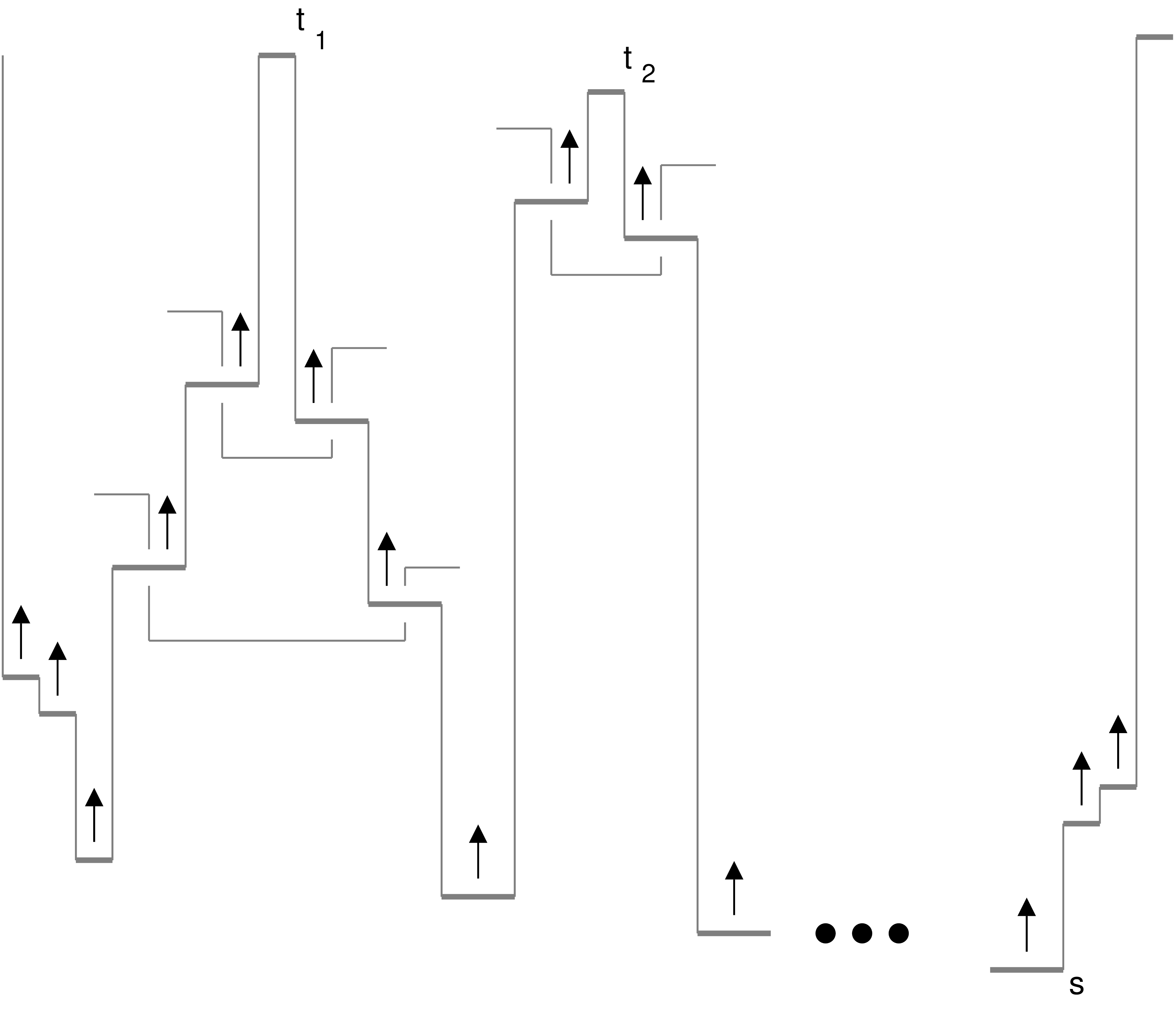}
\caption{The result of commuting each $t_i$ above $r_{yz}$. } \label{iso-t-done-fig}
\end{figure}

The final step is to commute the remaining segments of $s$ until their $y$-coordinates are greater than the greatest $y$-coordinate of $r_{yz}$. One can commute $y$-flats containing segments of $s$ without breaking the crossing conditions since the segments of $s$ which are part of crossings in the $(z,x)$-projection (i.e. the $t_i$'s) already have a larger $y$-coordinate than the greatest $y$-coordinate of $r_{yz}$. When commuting $y$-flats containing segments of $s$, no crossings are created or destroyed in the $(x,y)$-projection. Thus the remaining segments of $s$ can be commuted above $r_{yz}$ using a sequence of Reidemeister II and III moves in the $(y,z)$-projection. By Lemma \ref{lemma:Reidemeister}, there is a sequence of cube moves that commutes the segments of $s$ so that each one has a greater $y$-coordinate than the greatest $y$-coordinate of $r_{yz}$. 

The resulting cube diagram is $\Gamma^\prime$, and it represents a link $L^\prime$ that is strongly equivalent to $\widehat{L}$ in all $3$ coordinate plane link projections.
\end{proof}

A consequence of the Isotopy Lemma is the main  theorem of the paper: if two cube diagrams represent the same oriented link, then one can be transformed into the other through cube moves.  The strategy of the proof is to find a sequence of grid diagrams that takes the $(x,y)$-projection of the first cube diagram to the $(x,y)$-projection of the second and then modifying the sequence until the corresponding moves on the grid can be carried out in the cube.

\medskip

\begin{theorem}\label{Moves_Theorem}
Two cube diagrams correspond to ambient isotopic oriented links if and only if one can be obtained from the other by a finite sequence of cube stabilization and cube commutation moves.
\end{theorem}

\medskip

In \cite{Reid}, Reidemeister proved the following fact:  Two oriented piecewise linear links are ambient isotopic if and only if they can be connected through a sequence of vertex creation/destruction moves and triangle moves.  Recall that a triangle move simply replaces one segment of a piecewise linear link with two consecutive segments (or replaces two consecutive segments with a third segment) if the three segments are the boundary of a triangular region that does not intersect the link.  The vertex and triangle moves  have corresponding moves in cube diagrams.  For example, the cube stablization/destablization move is similar to the vertex creation/destruction move.  We will describe a  triangle-like move that can, in the end, be replaced by a sequence of cube moves using the Isotopy Lemma.

\medskip

First, we describe a method to transform one cube diagram to another using moves that are reminiscent of triangle moves, but which are less restrictive than cube commutations. In particular, the crossing conditions on the projections are ignored. Let $b$ be a $z$-cube bend in a $z$-flat $B$, such that $b$ is composed of two segments $b_1$ and $b_2$. Insert a copy of $B$ anywhere in the cube along the $z$-axis. Name the new copy $B'$ and suppose the translate of $b$ in $B'$ is named $b'$ and is composed of two segments $b_1'$ and $b_2'$. The parallel segments $b_1$ and $b_1^\prime$ form two sides of a rectangle $r_1$, and the parallel segments $b_2$ and $b_2^\prime$ form two sides of a rectangle $r_2$. Suppose the interiors of $r_1$ and $r_2$ are disjoint from the link (see Figure~\ref{cubebendmove}). Then the $z$-flat $B$ can be removed from the cube, while the $z$-flat $B'$ is kept in the cube. Call such a move a {\it $z$-cube bend move}. Similarly define {\it $x$-cube bend moves} and {\it $y$-cube bend moves}. Note that the marking conditions continue to hold but the crossing conditions may no longer be satisfied.
\begin{figure}[H]
\includegraphics[scale=.5]{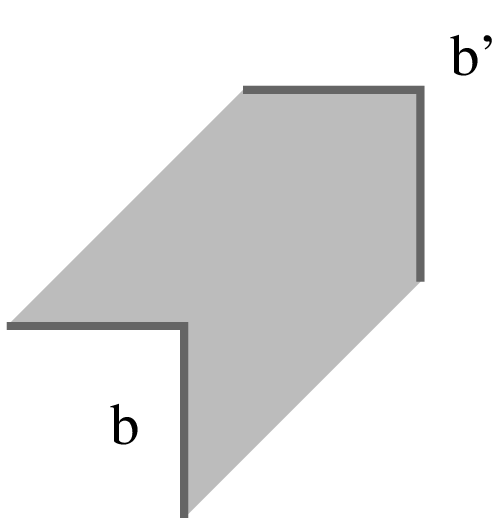}
\caption{If the shaded region is disjoint from the link, then $b$ can be replaced with $b'$.}
\label{cubebendmove}
\end{figure}

\medskip

\begin{lemma}
\label{cubebendlemma}
Let $\Gamma$ and $\Gamma'$ be two cube diagrams representing the same oriented link. Then $\Gamma'$ can be obtained from $\Gamma$ from a finite sequence of cube stabilizations and cube bend moves.
\end{lemma}
\begin{proof}
Let $G$ be the grid diagram associated to the $(x,y)$-projection of $\Gamma$ and let $G'$ be the grid diagram associated to the $(x,y)$-projection of $\Gamma'$. Then there is a sequence of grid diagram moves taking $G$ to $G'$ by Proposition~\ref{gridmovesprop}. This sequence can be used to induce a sequence of moves on the cube diagram taking $\Gamma$ to $\Gamma'$.

\medskip

There are two possible ways a commutation move on the $(x,y)$-projection of a cube diagram might not induce a cube bend move on the cube: either switching the corresponding flats in the  cube is not a cube bend move or switching the corresponding flats breaks the crossing conditions for the $(x,y)$-projection of the cube (see Figure~\ref{badcom3}). In general cube bend moves do not have to preserve the crossing conditions in any projection;  in what follows we consider only cases where the crossing conditions are preserved in the $(x, y)$-projection.
\begin{figure}[H]
\includegraphics[scale=.3]{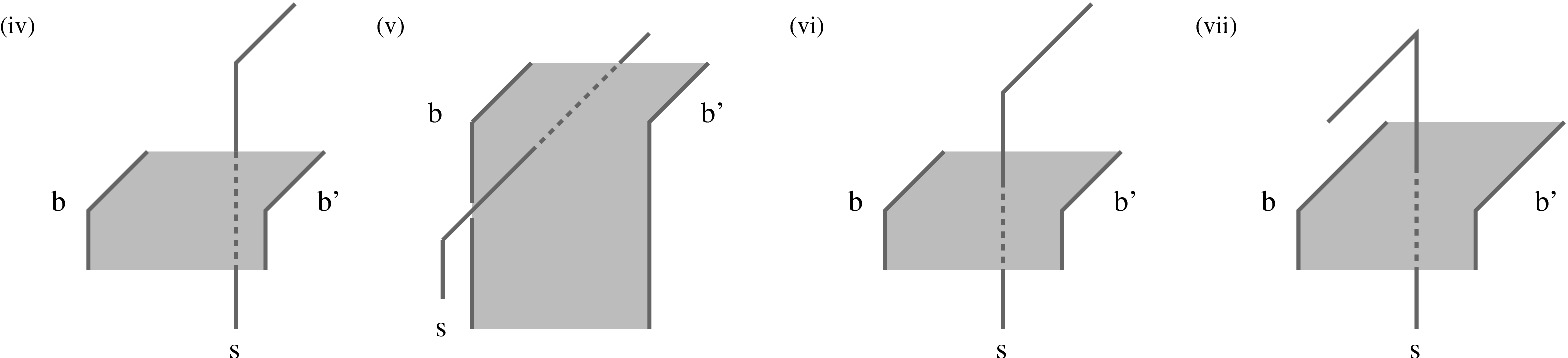}
\caption{Each of these figures depicts an attempted $x$-cube bend move where $b$ is replaced with $b'$. These cases are from the bottom row of  Figure~\ref{badcom2} where a commutation move on the grid diagram associated to the $(x,y)$-projection does not induce a cube bend move.}
\label{badcom3}
\end{figure}

Each of the $x$-cube  bend moves in Figure~\ref{badcom3} can be achieved if the $z$-flat containing the segment $s$ is moved via a $z$-cube bend move into the appropriate spot: in front of the shaded regions for cases (iv), (vi), and (vii), and behind the shaded region for (v).  Because of the crossing conditions in the $(x,y)$-projection, the segment $s$ is free to move where necessary as part of a $z$-cube bend move. However, the other segment in the $z$-cube bend containing $s$ may prevent the necessary $z$-cube bend move (see the first picture of Figure~\ref{untwistbends}).
\begin{figure}[H]
\includegraphics[scale=.25]{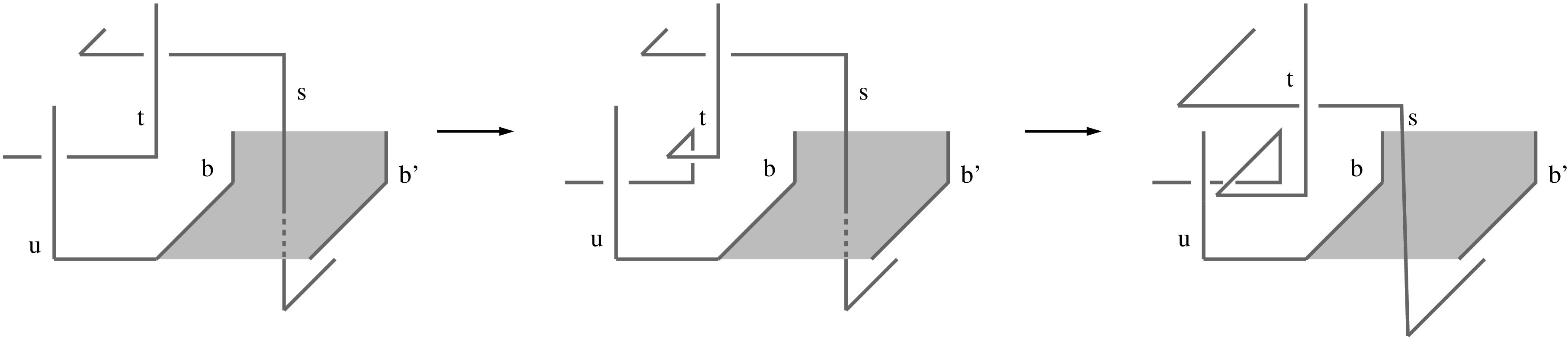}
\caption{The first picture illustrates why partial ordering is important to performing cube bend moves: the $z$-cube bend containing $t$ prevents the $z$-cube bend containing $s$ from being moved in front of the cube bend $u$.  In the second picture the cube bend $u$ is no longer greater than the cube bend containing $s$ (by stabilizing at $t$).  The third picture shows that $t$ and $s$ can  be moved forward. The result is a diagram where the $x$-cube bend move taking $b$ to $b'$ can be performed.}
\label{untwistbends}
\end{figure}

\medskip

Put a partial ordering on the $z$-flats as follows. Let $z_1$ and $z_2$ be two $z$-cube bends. If $z_1$ and $z_2$ intersect when projected to the $(x,y)$-plane and if $z_1$ has a greater $z$-coordinate than the $z$-coordinate of $z_2$, then $z_1>z_2$.  Suppose that we wish to perform an $x$-cube bend move that corresponds to a grid commutation in the $(x,y)$-projection.  Furthermore, in order to perform the $x$-cube bend move, suppose a $z$-cube bend $z_2$ is required to be moved so that its $z$-coordinate is greater than another $z$-cube bend $z_1$ (see the first picture of Figure~\ref{untwistbends}).  If $z_1>z_2$, then we say that $z_1$-cube bend and $z_2$-cube bend are {\em barriers} to performing the $x$-cube bend move.  If $z_1$ and $z_2$ are not comparable, then  $z$-cube bend moves can be performed to move $z_2$ so that its $z$-coordinate is greater than the $z$-coordinate of $z_1$. (If $z_2>z_1$, then the $z$-coordinate of $z_2$ is already greater then the $z$-coordinate of $z_1$.)

\medskip

For example, in Figure~\ref{untwistbends} the $z$-cube bend containing $s$ is less than the $z$-cube bend $u$, making the $z$-cube bend $u$ a barrier to moving $s$ into position where a $x$-cube bend move can be performed that moves $b$ to $b'$.  However, by stabilizing at $t$ (second picture), the two bends are no longer comparable, and the $z$-cube bends containing $s$ and $t$ can be moved together so that the $z$-coordinate of $s$ is greater than the $z$-coordinate $u$ (third picture).  In general, we can always stabilize so that two $z$-cube bends are no longer comparable.  Note that barriers to an $x$-cube bend move that correspond to a grid commutation move in the $(x,y)$-projection only involve the $z$-cube bends that contain or abut the $x$-cube bends being commuted in the $(x,y)$-projection.  Therefore, at most four comparisons need to be made for each such $x$-cube bend move.

\medskip

There is a similar notion of barriers for grid diagrams.  If $G$ is a grid diagram, partition the segments of $G$ into bends such that each bend in $G$ has an $X$ marking for a vertex.  Suppose that a partial ordering can be put on the bends such that a bend that crosses over another bend is considered greater than that bend.  (Partial orderings always exists after possibly stabilizing $G$, cf. Lemma~\ref{nobendtwistedlemma}.)  A grid commutation move on $G$ may or may not induce a partial ordering on the resulting grid.  If $G$ is a grid diagram that is the $(x,y)$-projection of a cube $\Gamma$ that minimally satisfies the marking conditions, and if a grid commutation move on $G$ induces a partial ordering on the resulting grid, then there are no barriers to making the $x$-cube bend move in $\Gamma$.

\medskip

Let $S:G=G_1 \ra G_2 \ra \cdots \ra G_m=G'$ be a sequence of grid moves taking $G$ to $G'$. Partition the segments of $G_i$ into bends such that each bend in $G_i$ has an $X$ marking for a vertex. The partition for $G=G_1$ corresponds to choosing $z$-bend partition in $\Gamma$, so the bends of $G_1$ can be partially ordered (similarly, $G_m=G'$ can be partially ordered). Suppose that for all moves $G_i\ra G_{i+1}$ in $S$, the partial ordering on $G_i$ induces a partial ordering on $G_{i+1}$.  Then $S$ induces a sequence of cube stabilizations/destabilizations and cube bend moves taking $\Gamma$ to $\Gamma''$, where the $(x,y)$-projection of $\Gamma''$ is $G'$ but $\Gamma'$ may not satisfy the crossing conditions in the $(y,z)$-projection or the $(z,x)$-projection.

\medskip

Thus, it remains to find such a sequence of grid moves.  Now let $S:G=G_1 \ra G_2 \ra \cdots \ra G_m=G'$ be any sequence of grid moves taking $G$ to $G'$ (in particular, the sequence given by Proposition~\ref{gridmovesprop} can be used). If a move $G_i\ra G_{i+1}$ in $S$ does not induce a partial ordering on $G_{i+1}$, then we alter the sequence so that it does. First, note that a stabilization move always induces a partial ordering on the resulting diagram, while destabilization moves and commutation moves may not.

\medskip

Starting with $G=G_1$, step down the sequence $S$ making the following modifications when necessary:  If a destabilization move does not induce a partial ordering (say by contracting two bends in $G_i$ to get one bend in $G_{i+1}$), then postpone the destabilization move until later as explained below. Furthermore, replace each commutation move $G_i \ra G_{i+1}$ with a sequence of moves where each move in the sequence induces a partial ordering on the next grid diagram in the sequence.  For example, consider the commutation shown in Figure~\ref{stabcommuteunstab}.   Let $x_1, x_2, y_1,$ and $y_2$ be the bends in $G_i$ that contain or abut the two segments $x$ and $y$ that are being commuted.  For each $x_i$ or $y_i$, stabilize at the vertex of the bend if both segments in the bend have length greater than 1.  Each bend in the new set of bends contains a length 1 segment and are untwisted.  Therefore, if a bend in this set is less than another bend in the grid diagram, then it is less than all other bends in the grid diagram to which it can be compared.  Similarly, if a bend in the set is greater than another bend, then it is greater than all other comparable bends.

\medskip

Suppose that the segments $x$ and $y$ have length greater than 1 (see second picture).  Both bends containing them are greater than all other comparable bends, and therefore one can assign the bend containing $y$ to be greater than the bend containing $x$ after they are commuted. Similarly, if the length of $x$ is 1, then the bend containing $x$ is less than all other comparable bends and the bend containing $y$ is greater than all comparable bends. Therefore the bend containing $y$ can be assigned to be greater than the bend containing $x$ after they are commuted.  Similar reasoning applies to the other bends in the second picture of Figure~\ref{stabcommuteunstab}.

\medskip

Thus, the segments $x$ and $y$ and their associated stabilized bends can be commuted while continuing to induce a partial order on the grids after each move (See third picture of Figure~\ref{stabcommuteunstab}).  Finally, perform destabilizations if they induce a partial ordering on the resulting diagram (see fourth picture).  If a previous stabilization in this step could not be destabilized, then move it with its associated segments as described below.

\begin{figure}[H]
\includegraphics[scale=.2]{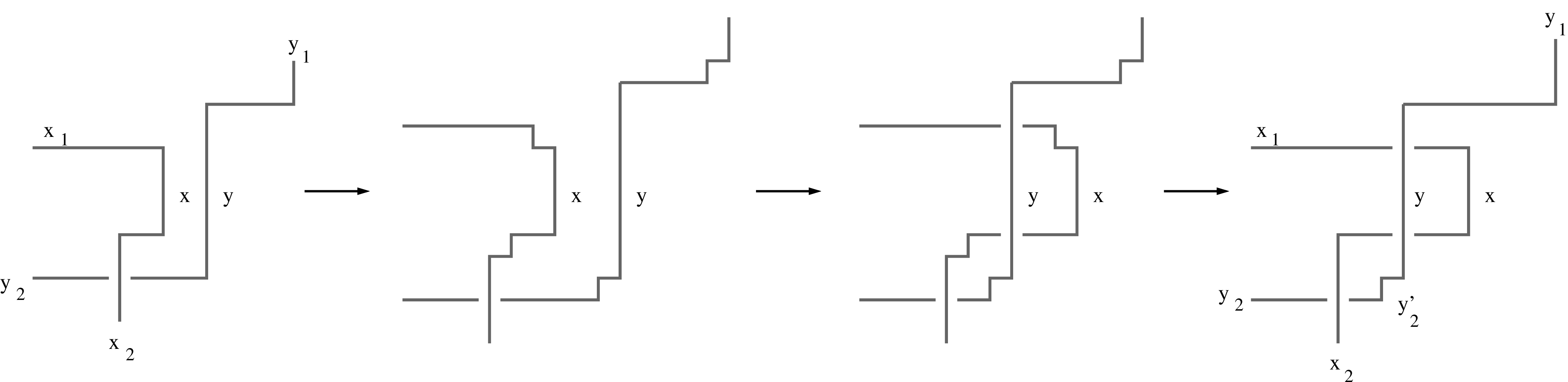}
\caption{If $x$ and $y$ are commuted in the first picture, then the resulting grid diagram (not shown) does not have a partial ordering because $x_2$ would be both greater than and lesser than $y_2$. To commute $x$ and $y$ in a way that induces a partial ordering, stabilize each of the bends, commute separately, and destabilize when possible.}
\label{stabcommuteunstab}
\end{figure}

Note that the three commutation cases where there are only 3 bends is done similarly.  The case where the two segments $x$ and $y$ are disjoint when projected to $y$-axis can be commuted directly; that commutation automatically induces a partial order on the resulting grid diagram.

\medskip

Next we deal with modifying the remaining sequence to take into account stabilizations that were either postponed or created in the moves above.  In both cases, a single bend $a=a_1\cup a_2$ in the grid diagram in the original sequence has been replaced by two bends $a_1\cup a'_1$ and $a_2 \cup a'_2$ such that $a'_1$ and $a'_2$ are of length 1 in the new sequence. If in the original sequence, the segment $a_1$ is commuted, then in the modified sequence, the segments $a_1,$ $a'_1$ and $a'_2$ are commuted as a group so that, at the end of the group of moves, $a'_1$ and $a'_2$ are still length 1 (see Figure~\ref{invproof3}).
\begin{figure}[H]
\includegraphics[scale=.4]{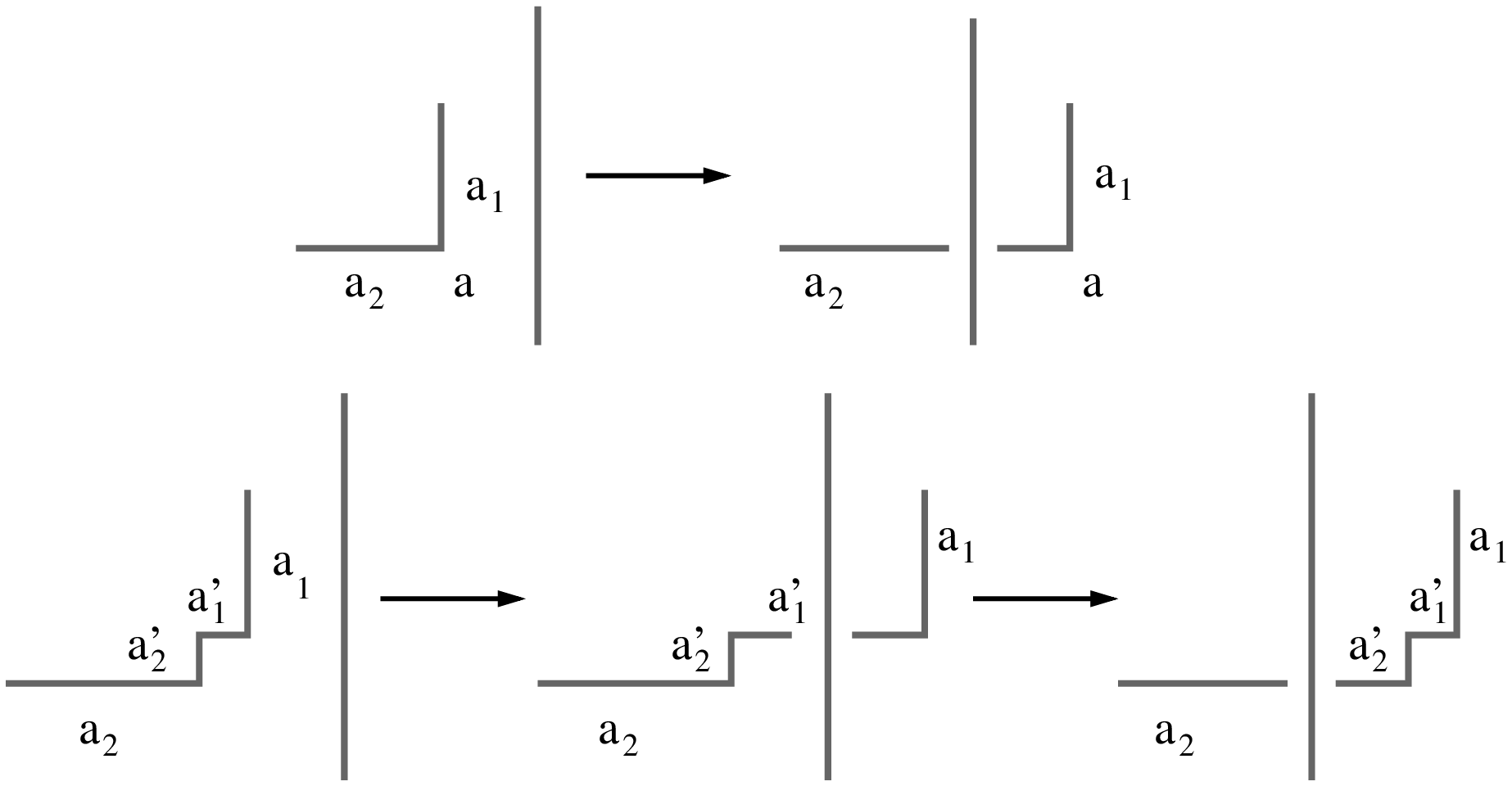}
\caption{A commutation move before stabilization and two moves commuting the group of segments after a stabilization.}
\label{invproof3}
\end{figure}
Observe that during the commutation as a group, either segment $a'_1$ or $a'_2$ is always of length $1$. For the same reasons as described above, the modified sequence of grid moves continues to induce a partial ordering after each move. If either segments $a_1$ or $a_2$ are part of a destabilization move later in the sequence, then also destabilize the bend $a'_1\cup a'_2$ if the move induces a partial ordering on the resulting grid diagram (if not, destabilize $a'_1\cup a'_2$ and postpone the destabilization of the bend containing $a_1$ or $a_2$). Finally, the sequence must be modified by adding a destabilization move (destabilizing $a'_1 \cup a'_2$  which then contracts $a_1\cup a'_1 \cup a'_2 \cup a_2$ into $a$) once there are no more commutations involving $a$ in the original sequence and once the destabilization induces a partial ordering on the resulting grid diagram.

\medskip

We have shown that there exists a sequence of grid moves on $G$ that induces a sequence of cube stabilizations/destabilizations and cube bend moves taking $\Gamma$ to $\Gamma''$.  Since the $(x,y)$-projections of both $\Gamma'$ and  $\Gamma''$ are $G'$, the partial orderings on the $z$-flats agree. Therefore $\Gamma''$ can be transformed into $\Gamma'$ through a sequence of $z$-cube bend moves.
\end{proof}

\medskip

\begin{proof}[Proof of Theorem \ref{Moves_Theorem}]
Let $\Gamma$ and $\Gamma'$ be two cube diagrams representing the same oriented link.
Lemma \ref{cubebendlemma} states that there is a sequence $\Sigma:\Gamma =\Gamma_1 \ra \Gamma_2 \cdots \ra \Gamma_n=\Gamma'$ of cube stabilizations/destabilizations and cube bend moves taking $\Gamma$ to $\Gamma'$. It remains to show that this sequence can be modified to only contain cube moves (not cube bend moves).

\medskip

Starting at the beginning of the sequence, we use the Isotopy Lemma to replace each cube bend move in $\Sigma$ one at a time with a subsequence of  cube moves.  Some care must be taken---the Isotopy Lemma can and will introduce new stabilizations, which increases the cube size and introduces new segments in the link.  Number each segment in the original $\Gamma$ and each new segment introduced by a stabilization move in $\Sigma$; let $N$ be the total number of segments. Each time the Isotopy Lemma introduces a new set of segments through a stabilization, assign those segments the same number as one of the adjacent segments (so that the number becomes a label for the set of connected  segments).  The sequence $\Sigma$ can then be modified to incorporate the new segments.  For example, if the original numbered segment is moved by a cube bend move later in the sequence, we move the entire set of segments with the same number as a unit, as allowed by the Isotopy Lemma.  In this way we can move down the sequence, replacing each cube bend move with a subsequence of cube moves, assigning numbers to any new segments introduced, and introducing new cube bend moves later in the sequence to account for moving the new segments. The number of segments counted or introduced in $\Sigma$ is $N$ and there are only a finite number of moves done to each of those segments in the sequence $\Sigma$. Also, each stabilization introduced by the Isotopy Lemma must be followed later in $\Sigma$ by a destabilization.  Therefore, the replacement of cube bend moves with subsequences of cube moves must eventually come to an end. Thus $\Sigma$ induces a sequence of cube moves that takes $\Gamma$ to $\Gamma'$.
\end{proof}

\medskip

To check for invariants of links, one need only check that a property of a cube diagram remains invariant under the generating cube stabilization move (Figure \ref{cubestab}) and the 4 generating cube commutation moves (Figure \ref{cubecom}).

\medskip

\begin{Corollary}\label{Moves_Corollary} Any property of a cube diagram that does not change under the 5 generating cube commutation and cube stabilization moves is an invariant of the link.
\end{Corollary}

\medskip

A fact that emerged from the proof of Theorem~\ref{Moves_Theorem} is the following scholium:

\begin{scholium}
Let $G$ and $G'$ be two grid diagrams representing the same oriented link.  Let $\mathcal{B}$  be the partition of bends of a grid diagram such that the vertex of each bend is an $X$.  If there exists a partial ordering on the bends in $\mathcal{B}$ and $\mathcal{B'}$, then there exists a sequence $S:G{=}G_1 \ra G_2 \ra \cdots\ra G_n{=}G'$ of stabilization and commutation moves such that, for each transformation $G_i \ra G_{i+1}$, the partial ordering on $\mathcal{B}_i$ induces a partial ordering on $\mathcal{B}_{i+1}$.
\end{scholium}

\medskip
\section{A cube homology theory}
\label{cubehomology}
\medskip

As an example of Theorem~\ref{Moves_Theorem} and Corollary~\ref{Moves_Corollary} we describe a homology that is clearly invariant under cube moves, and hence
a knot invariant.  One nice aspect of this example is that it can be shown to be a knot invariant using other methods.  Thus, it can be viewed as a check of the main theorem of this paper.  Another nice aspect of the homology is that it gives an invariant of a 3-dimensional representation of the knot.  One hopes that this viewpoint may lead to new insights into knot Floer homology.

\medskip

Given a cube diagram $\Gamma$ and a choice of two projections, we associate to $\Gamma$ a bigraded chain complex. The generators for this complex do not depend on the choice of projections; however, the gradings and differential do. Since many of the following constructions depend on the choice of two projections, we take some time to gather notation. If an object depends on only one projection of the cube, it is indexed by the two coordinate vectors which span that projection (i.e. if an object $\mathcal{O}$ depends only on the $(x,y)$-projection, then it will be represented as $\mathcal{O}_{xy}$). However, if an object depends on two projections simultaneously, then it will be indexed by the shared basis vector of the two projections (i.e. if an object $\mathcal{O}$ depends on both the $(x,y)$-projection and $(y,z)$-projection, then it will be represented as $\mathcal{O}_y$).

\medskip

Let $P$ be the set of integer lattice points in $\Gamma$ with no coordinate equal to $n$, the size of $\Gamma$. A \textit{cube state ${\bf s}$ of $\Gamma$} is a subset of $P$ such that

\begin{itemize}
\item $|{\bf s}|=n$, and\\

\item no two points in ${\bf s}$ lie on the same face of any flat in $\Gamma$.
\end{itemize}

\medskip

Each cube state is represented by a collection of dots on the cube diagram. See Figure \ref{cubestate} for an example of a cube state.
\begin{figure}[H]
\includegraphics[scale=.3]{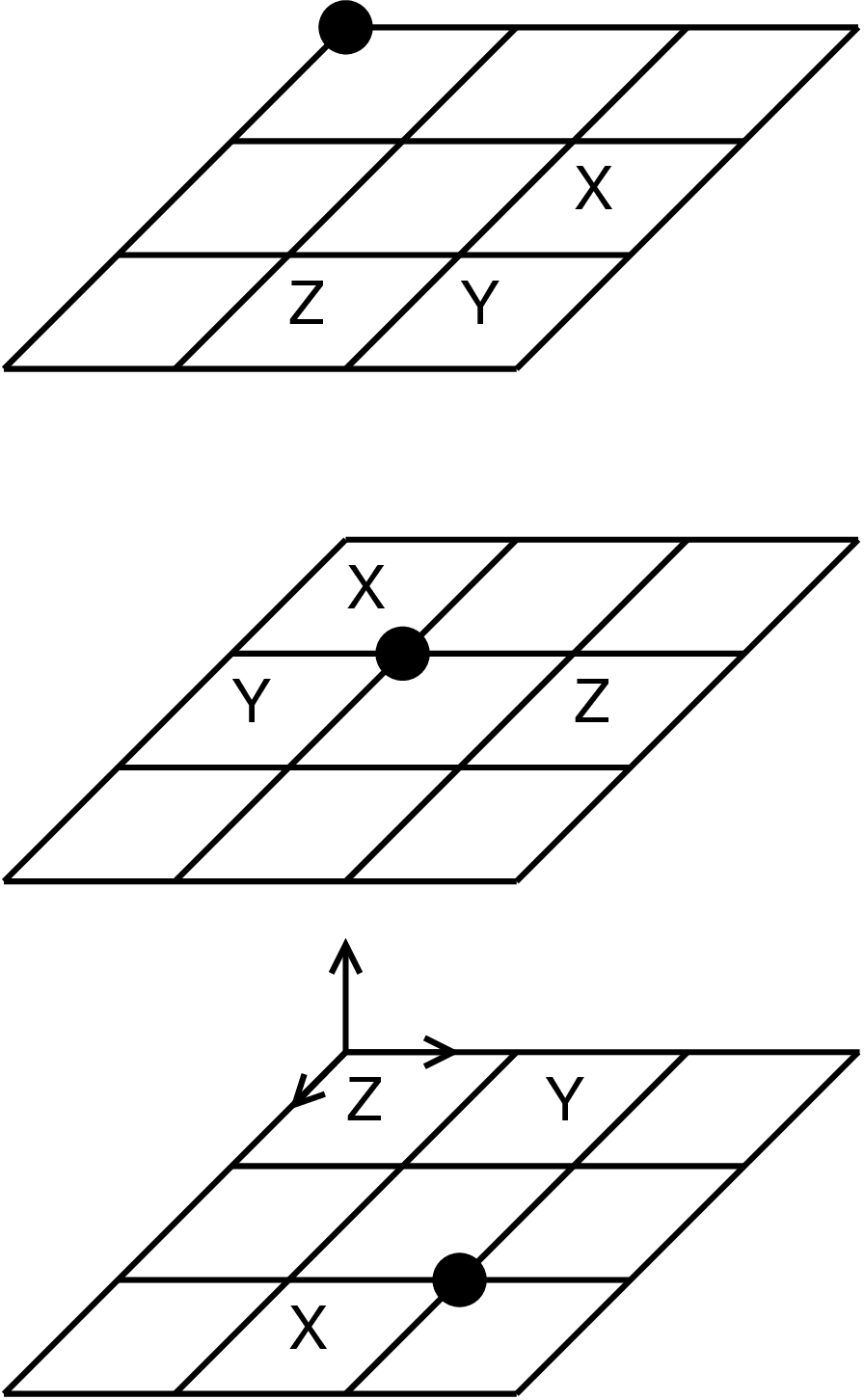}
\caption{A cube state for a size 3 cube diagram.} \label{cubestate}
\end{figure}
Denote the set of all cube states by ${\bf S}$.

\smallskip

In order to define the gradings, we define several functions. Let $A$ and $B$ be two finite sets of points in $\mathbb{R}^3$.

\begin{itemize}
\item $I_{xy}(A,B)$ is defined to be the number of pairs $(a_1,a_2,a_3)\in A$ and $(b_1,b_2,b_3)\in B$ such that $a_1<b_1$ and $a_2<b_2$;\\

\item $I_{yz}(A,B)$ is defined to be the number of pairs $(a_1,a_2,a_3)\in A$ and $(b_1,b_2,b_3)\in B$ such that $a_2<b_2$ and $a_3<b_3$;\\

\item $I_{zx}(A,B)$ is defined to be the number of pairs $(a_1,a_2,a_3)\in A$ and
$(b_1,b_2,b_3)\in B$ such that $a_1<b_1$ and $a_3<b_3$.
\end{itemize}

\smallskip

Define $J_{xy}(A,B)=\left(I_{xy}(A,B)+I_{xy}(B,A)\right)/2$, and similarly define $J_{yz}$ and $J_{zx}$. Moreover, these functions can be extended bilinearly over formal sums and differences. Figure \ref{helper} illustrates a geometric interpretation for one of the functions.
\begin{figure}[H]
\includegraphics[scale=.3]{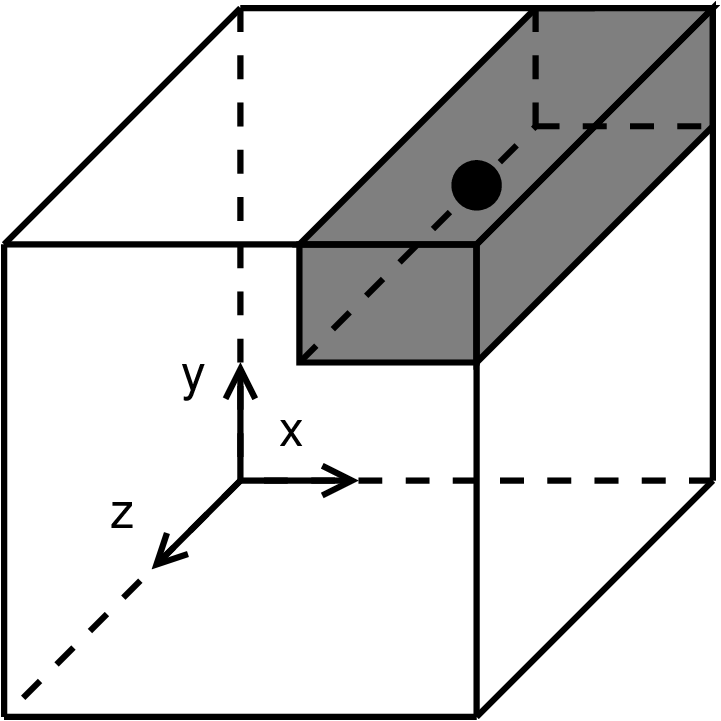}
\caption{For each point in $A$, the function $I_{xy}(A,B)$ counts the number of points of $B$ which are in the above cuboid.} \label{helper}
\end{figure}

\medskip

Define the {\it Maslov gradings} by

\begin{itemize}
    \item $M_{xy}({\bf s}) = J_{xy}({\bf s}-\mathcal{X},{\bf s}-\mathcal{X})+1$,\\

    \item $M_{yz}({\bf s}) = J_{yz}({\bf s}-\mathcal{Y},{\bf s}-\mathcal{Y})+1$, and\\

    \item $M_{zx}({\bf s}) = J_{zx}({\bf s}-\mathcal{Z},{\bf s}-\mathcal{Z})+1.$
\end{itemize}
Define the {\it Alexander gradings} by

\begin{itemize}
    \item $A_{xy}({\bf s}) = J_{xy}({\bf s}-\frac{1}{2}(\mathcal{Z}+\mathcal{X}),\mathcal{Z}-\mathcal{X})+\frac{n-1}{2}$,\\

    \item $A_{yz}({\bf s}) = J_{yz}({\bf s}-\frac{1}{2}(\mathcal{X}+\mathcal{Y}),\mathcal{X}-\mathcal{Y})+\frac{n-1}{2}$, and\\

    \item $A_{zx}({\bf s}) = J_{zx}({\bf s}-\frac{1}{2}(\mathcal{Y}+\mathcal{Z}),\mathcal{Y}-\mathcal{Z})+\frac{n-1}{2}$.
\end{itemize}

\medskip

The $3$-torus $\mathbb{T}^3\cong S^1\times S^1\times S^1$ is a natural quotient of the cube diagram $\Gamma$. An {\em $(x,y)$-cylinder} $c$ in $\Gamma$ (viewed as on the $3$-torus) is a cuboid with integer vertices such that all edges parallel to the $z$-axis are length $n$. Similarly, define $(y,z)$-cylinders and $(z,x)$-cylinders. Figure \ref{cylex} shows an example of an $(x,y)$-cylinder in a cube diagram. We think of $\Gamma$ as a $3$-torus in order to conveniently define cylinders. With the definition of a cylinder understood, we will continue to work with the cube $\Gamma$, not the $3$-torus.

\medskip

For the remainder of the section, we choose the $(x,y)$-projection and $(y,z)$-projection for a cube diagram $\Gamma$. Similar constructions hold for any other choice of two projections. Let $\pi_{xy}:\mathbb{R}^3\to\mathbb{R}^2$ be the projection map to the $(x,y)$-plane and $\pi_{yz}:\mathbb{R}^3\to\mathbb{R}^2$ be the projection map to the $(y,z)$-plane.
Let ${\bf s}$ and ${\bf t}$ be two cube states.  A {\it $(x,y)$-cylinder $c$ connecting ${\bf s}$ to ${\bf t}$} is an $(x,y)$-cylinder such that:

\begin{enumerate}
\item the cube states ${\bf s}$ and ${\bf t}$ differ at exactly two points, $s_1,s_2\in{\bf s}$ and $t_1,t_2\in{\bf t}$,\\

\item the $(y,z)$-projections of ${\bf s}$ and ${\bf t}$ agree, that is $\pi_{yz}({\bf s}) = \pi_{yz}({\bf t})$,\\

\item the four points $\pi_{xy}(s_1),\pi_{xy}(s_2),\pi_{xy}(t_1)$ and $\pi_{xy}(t_2)$ are corners of the rectangle $r=\pi_{xy}(c)$, and\\

\item all segments parallel to the $x$-axis in the boundary of $r$ oriented by starting at a point of $\pi_{xy}({\bf s})$ and ending at a point of $\pi_{xy}({\bf t})$ must have the same orientation as the orientation that the boundary of $r$ inherits from the $(x,y)$-projection.
\end{enumerate}

\medskip

Figure \ref{cylex} shows an example of an $(x,y)$-cylinder connecting ${\bf s}$ to ${\bf t}$.
\begin{figure}[H]
\includegraphics[scale=.3]{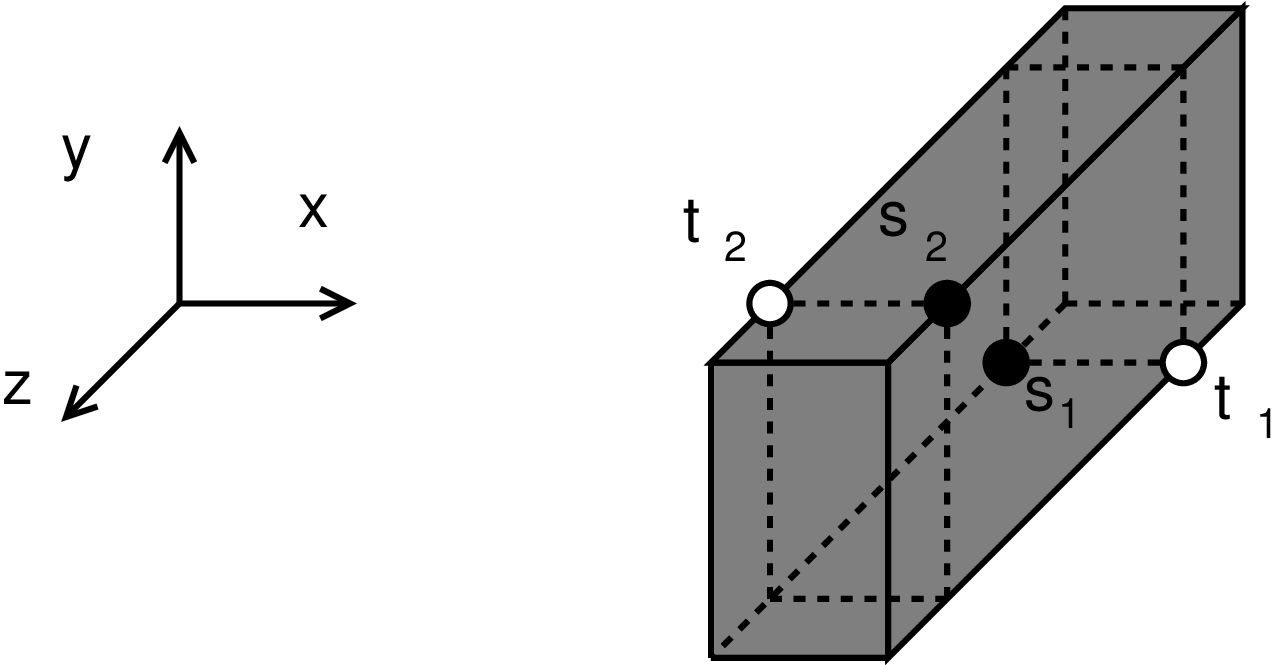}
\caption{An example of an $(x,y)$-cylinder connecting ${\bf s}$ to ${\bf t}$.} \label{cylex}
\end{figure}

Similarly define a $(y,z)$-cylinder from ${\bf s}$ to ${\bf t}$ by interchanging the roles of $\pi_{xy}$ and $\pi_{yz}$ in the definition and using segments parallel to the $y$-axis in condition 4. Observe that this definition depends on the choice of the two projections. For each choice of two projections, there are new definitions for the appropriate types of cylinders connecting two states.

\medskip

A cylinder $c$ connecting ${\bf s}$ to ${\bf t}$ is said to be {\it empty} if
Int$(c)\cap {\bf s}=\emptyset$ or equivalently if Int$(c)\cap {\bf t}=\emptyset$.
Let $\mathrm{Cyl}_{xy}^\circ({\bf s},{\bf t})$ denote the set of empty $(x,y)$-cylinders from ${\bf s}$ to ${\bf t}$ and $\mathrm{Cyl}_{yz}^\circ({\bf s},{\bf t})$ denote the set of empty $(y,z)$-cylinders from ${\bf s}$ to ${\bf t}$.

\medskip

Label the markings of $\Gamma$ so that $\mathcal{X}=\{X_i\}_{i=1}^n$, $\mathcal{Y}=\{Y_i\}_{i=1}^n$, and $\mathcal{Z}=\{Z_i\}_{i=1}^n$ and define three sets of corresponding variables $\{\mathtt{\bf \bf X}_i\}_{i=1}^n, \{\mathtt{\bf Y}_i\}_{i=1}^n,$ and $\{\mathtt{\bf Z}_i\}_{i=1}^n$. Define $R_y$ to be the polynomial algebra over $\mathbb{Z}/2\mathbb{Z}$ generated by the variables $\{\mathtt{\bf X}_i\}_{i=1}^n$ and $\{\mathtt{\bf Y}_i\}_{i=1}^n$. The constant terms of $R_y$ are in the zero grading of all four gradings: $M_{xy}$, $M_{yz}$, $A_{xy},$ and $A_{yz}$. The $\mathtt{\bf X}_i$ have $M_{xy}$ grading $-2$, $A_{xy}$ grading $-1$, and all other gradings $0$. The $\mathtt{\bf Y}_i$ have $M_{yz}$ grading $-2$, $A_{yz}$ grading $-1$, and all other gradings $0$. Similarly define $R_z$ and $R_x$. Define $C_y^-(\Gamma)$ to be the free $R_y$-module with generating set ${\bf S}$. One can similarly define $C_z^-(\Gamma)$ and $C_x^-(\Gamma)$.

\medskip

For a cylinder $c$ in $\Gamma$, let $X_i(c)$ count the number of times the marking $X_i$ appears inside $c$. Similarly define $Y_i(c)$ and $Z_i(c)$. The map $\partial_y^-:C_y^-(\Gamma)\to C_y^-(\Gamma)$ is the differential of the chain complex and is defined by
$$\partial_y^-({\bf s}) = \sum_{{\bf t}\in {\bf S}} \left(\sum_{c\in \text{Cyl}_{xy}^\circ({\bf s},{\bf t})}\mathtt{\bf X}_1^{X_1(c)}\cdots \mathtt{\bf X}_n^{X_n(c)}\cdot {\bf t}+\sum_{c\in \text{Cyl}_{yz}^\circ({\bf s},{\bf t})}\mathtt{\bf Y}_1^{Y_1(c)}\cdots \mathtt{\bf Y}_n^{Y_n(c)}\cdot {\bf t}\right).$$
Similarly define the differentials $\partial_z^-$ and $\partial_x^-$ for the chain complexes $C_z^-(\Gamma)$ and $C_x^-(\Gamma)$ respectively.

\medskip

Define the {\em Maslov grading} on $C_y^-(\Gamma)$ by
$$M_y({\bf s}) = M_{xy}({\bf s}) + M_{yz}({\bf s}),$$
and define the {\it Alexander grading} by
$$A_y({\bf s}) = A_{xy}({\bf s}) + A_{yz}({\bf s}).$$
One can similarly define $M_z,M_x,A_z$, and $A_x$.

\medskip

Before we prove that the machinery created above forms a chain complex and gives a link invariant, we will discuss the relationship between $(C_y^-(\Gamma),\partial_y^-)$ and a chain complex coming from grid diagrams. In \cite{mos} and \cite{most}, a combinatorial description of knot Floer homology is given using grid diagrams. This construction uses a grid diagram to encode the Heegaard diagram, and the count of pseudo-holomorphic disks is given by counting rectangles in the grid diagram. For clarity, we review the construction of \cite{most}.

\medskip

Let $G$ be a grid diagram. Transfer $G$ to the torus by gluing the top and bottom edges and gluing the leftmost and rightmost edges. The grid lines then become grid circles. The resulting diagram is called a {\it toroidal grid diagram} or just grid diagram.

\medskip

Each grid diagram $G$ has an associated chain complex $(C^-(G),\partial^-)$. Let $R$ be the polynomial algebra over $\mathbb{Z}/2\mathbb{Z}$ generated by the set of elements $\{\mathtt{\bf U}_i\}_{i=1}^n$ that are in one-to-one correspondence with $\mathbb{O}=\{O_i\}_{i=1}^n$. The chain complex $C(G)^-$ is generated by the set of states $S_{\text{grid}}(G)$. A state of $G$ is an $n$-tuple of intersection points of the horizontal and vertical circles satisfying the condition that no horizontal (or equivalently no vertical) circle contains more than one intersection point.

\medskip

The complex $C(G)$ has a Maslov grading, given by a function $M:S_{\text{grid}}(G)\to\mathbb{Z}$, and an Alexander filtration level, given by $A:S_{\text{grid}}(G)\to\mathbb{Z}$ if $L$ has an odd number of components or $A:S_{\text{grid}}(G)\to\mathbb{Z}+\frac{1}{2}$ if $L$ has an even number of components. To define these gradings, we use some auxiliary functions. Let $A$ and $B$ be finite sets of points in $\mathbb{R}^2$. Define $I(A,B)$ to be the number of pairs $(a_1,a_2)\in A$ and $(b_1,b_2)\in B$ with $a_1<b_1$ and $a_2<b_2$. Then define $J(A,B)=(I(A,B)+I(B,A))/2$.

\medskip

Take a fundamental domain $[0,n)\times [0,n)$ for the toroidal grid diagram. View a state ${\bf s}\in S_{\text{grid}}(G)$ as a collection of points with integer coordinates, and also view $\mathbb{O}$ and $\mathbb{X}$ as sets of points with half integer coordinates. Then define
$$M({\bf s})=J({\bf s}-\mathbb{O},{\bf s}-\mathbb{O})+1,$$
where $J$ is extended bilinearly over formal sums and differences. Define
$$A({\bf s})=J({\bf s}-\frac{1}{2}(\mathbb{X}+\mathbb{O}),\mathbb{X}-\mathbb{O})+\frac{n-1}{2}.$$

\medskip

The differential $\partial^-$ of this chain complex counts rectangles between two states. Let ${\bf s}$ and ${\bf t}$ be states of a grid diagram $G$. An embedded rectangle $r$ in $G$ connects {\bf s} to {\bf t} if

\begin{itemize}
\item ${\bf s}$ and ${\bf t}$ differ along exactly two horizontal circles,\\

\item all four corners of $r$ are points in ${\bf s}\cup{\bf t}$,\\

\item if we traverse horizontal segments of $r$ in the direction indicated by the orientation inherited from the torus, then the segment is oriented from ${\bf s}$ to ${\bf t}$.
\end{itemize}

\medskip

A rectangle $r$ is {\it empty} if Int$(r)\cap{\bf s}=\emptyset$. The set of all empty rectangles connecting ${\bf s}$ to ${\bf t}$ is denoted Rect$^\circ({\bf s},{\bf t})$. Define the differential by
$$\partial({\bf s})=\sum_{{\bf t}\in S_{\text{grid}}(G)}\sum_{r\in\mathrm{Rect}^\circ({\bf s},{\bf t})}\mathtt{\bf U}_1^{O_1(r)}\cdots \mathtt{\bf U}_n^{O_n(r)}\cdot{\bf t},$$
where $O_i(r)$ is the count of how many times the marking $O_i$ appears in $r$.

\medskip

In \cite{most}, the following lemma is proved
\begin{lemma}
Suppose that $O_i$ and $O_k$ correspond to the same link component of $L$. Then multiplication by $\mathtt{\bf U}_i$ is filtered chain homotopic to multiplication by $\mathtt{\bf U}_k$.
\end{lemma}
This lemma allows us to view the homology of the complex $C^-(G)$ as a module over $\mathbb{Z}/2\mathbb{Z}[\mathtt{\bf U}_1,\dots,\mathtt{\bf U}_l]$, where $\mathtt{\bf U}_1,\dots,\mathtt{\bf U}_l$ correspond to $l$ different $O$-markings, each in a different component in the link.

\medskip

This construction gives a well-defined chain complex and leads to the following theorem of Manolescu, Ozsv\'ath, and Sarkar \cite{mos}.
\begin{theorem}
\label{MOS} Let $G$ be a grid presentation of a knot $K$. The data $(C^-(G),\partial^-)$ is a chain complex for the Heegaard-Floer homology $CF^-(S^3)$, with grading induced by $M$, and the filtration level induced by $A$ coincides with the link filtration of $CF^-(S^3)$.
\end{theorem}
Also, if each of the $\mathtt{\bf U}_i$ variables is set to $0$, this construction results in the ``hat" version of knot Floer homology. More specifically
$$H_*(C/\{\mathtt{\bf U}_i=0\}_{i=1}^n)\cong\widehat{HFK}(L)\otimes V^{n-1},$$
where $V$ is the two-dimensional bigraded vector space spanned by one generator in bigrading $(-1,-1)$ and one generator in bigrading $(0,0)$.

\medskip

Now that the details of the two chain complexes $(C_y^-(\Gamma),\partial_y^-)$ and $(C^-(G),\partial^-)$ have been presented, we develop the relationship between the complex associated to $\Gamma$ and the complexes associated to its projections.

\medskip

\begin{proposition}
\label{staterel} Let ${\bf s}$ be a cube state on a cube diagram $\Gamma$. Let $G_{xy}$ be the grid diagram associated to the $(x,y)$-projection. Then $\pi_{xy}({\bf s})$ is a state on $G_{xy}$. Moreover, $M_{xy}({\bf s})=M(\pi_{xy}({\bf s}))$ and $A_{xy}({\bf s})=A(\pi_{xy}({\bf s}))$. Similar statements hold for the grid diagrams associated to the $(y,z)$-projection and $(z,x)$-projection.
\end{proposition}
\medskip

\begin{proof}
A cube state can be viewed as a set of three permutations of the set $\{0,1,\dots,n-1\}$. The first permutation gives the $x$-coordinates of each point in ${\bf s}$, the second permutation gives the $y$-coordinates of each point in ${\bf s}$, and the third permutation gives the $z$-coordinates of each point in ${\bf s}$. Similarly, a state in a grid diagram can be seen as two permutations of $\{0,1,\dots,n-1\}$, one gives the $x$-coordinate of each point in the state and one gives the $y$-coordinate of each point. Since ${\bf s}$ is equivalent to three permutations, it follows that $\pi_{xy}({\bf s})$ is equivalent to two permutations. Thus $\pi_{xy}({\bf s})$ is a state in the grid diagram $G$.

\medskip

If $A$ and $B$ are finite sets of points in $\mathbb{R}^3$ and $\pi_{xy}:\mathbb{R}^3\to \mathbb{R}^2$ is a projection map to the $(x,y)$-coordinate plane, then $J_{xy}(A,B)=J(\pi_{xy}(A),\pi_{xy}(B))$. The grading equivalences follow from this fact.
\end{proof}

\medskip

In fact, the set of cube states can be built out of the sets of grid states for two projections. Also, the chain complex associated to a cube diagram can be constructed from the chain complexes associated to the grid diagrams coming from two of the projections.

\medskip

\begin{theorem}
\label{compthm}
Let $\Gamma$ be a cube diagram and $G_{xy}$ and $G_{yz}$ be grid diagrams associated to the $(x,y)$-projection and $(y,z)$-projection of $\Gamma$. Let $(C^-(G_{xy}),\partial_{xy}^-)$ and $(C^-(G_{yz}),\partial_{yz}^-)$ be the  chain complexes associated to $G_{xy}$ and $G_{yz}$ respectively. Then
$$(C_y^-(\Gamma),\partial_y^-)\cong(C^-(G_{xy}),\partial_{xy}^-)\otimes(C^-(G_{yz}),\partial_{yz}^-).$$
A similar statement holds for $C_z^-(\Gamma)$ and $C_x^-(\Gamma)$.
\end{theorem}

\medskip

\begin{proof}
Let ${\bf s}\in S_{\text{grid}}(G_{xy})$ be a grid state for $G_{xy}$ and ${\bf t}\in S_{\text{grid}}(G_{yz})$ be a grid state for $G_{yz}$. Then ${\bf s}$ can be written as ${\bf s}=\{(x_1,y_1),\dots,(x_n,y_n)\}$, and ${\bf t}$ can be written as ${\bf t}=\{(y_1,z_1),\dots,(y_n,z_n)\}$ where $\{x_i\}=\{y_i\}=\{z_i\}=\{0,1,\dots,n-1\}$. Define a map $\psi_y:C^-(G_{xy})\otimes C^-(G_{yz})\to C_y^-(\Gamma)$ by $s\otimes t\mapsto \{(x_1,y_1,z_1),\dots,(x_n,y_n,z_n)\}$. The map $\psi_y$ is a bijection on the generating set and thus extends to an isomorphism. The map is extended so that the $\mathtt{\bf U}_i$ variables coming from the $(x,y)$-projection are sent to the $\mathtt{\bf X}_i$ variables and the $\mathtt{\bf U}_i$ variables coming from the $(y,z)$-projection are sent to the $\mathtt{\bf Y}_i$ variables. Since $\pi_{xy}(\psi_y({\bf s}\otimes {\bf t}))={\bf s}$ and $\pi_{yz}(\psi_y({\bf s}\otimes {\bf t}))={\bf t}$, Proposition \ref{staterel} implies that $\psi_y$ preserves the gradings. Thus $C_y^-(\Gamma)\cong C^-(G_{xy})\otimes C^-(G_{yz})$ as $R_y$-modules.

It remains to show that $\partial_y^-(\psi_y({\bf s}\otimes{\bf t}))=\psi_y(\partial_{xy}^-\otimes\partial_{yz}^-({\bf s}\otimes {\bf t}))$. Observe that
$$\partial_{xy}^-\otimes\partial_{yz}^-({\bf s}\otimes {\bf t})  =  \partial_{xy}^-({\bf s})\otimes {\bf t} + {\bf s}\otimes\partial_{yz}^-({\bf t}).$$
The summand $\partial_{xy}^-({\bf s})\otimes {\bf t}$ counts empty rectangles in $G_{xy}$ connecting ${\bf s}$ to other states. An empty rectangle connecting ${\bf s}$ to some other state occurs in $G_{xy}$ precisely when $\Gamma$ has an empty $(x,y)$-cylinder connecting $\psi_y({\bf s}\otimes {\bf t})$ to some other cube state. Moreover, if the empty rectangle connecting ${\bf s}$ to ${\bf s^\prime}$ contains some marking $O_i$, then the corresponding empty $(x,y)$-cylinder contains the marking $X_i$. Therefore the coefficients of the states in the sum  agree (up to the isomorphism described above). A similar statement is true for the summand ${\bf s}\otimes\partial_{yz}^-({\bf t})$ replacing rectangles in $G_{xy}$ with rectangles in $G_{yz}$ and $(x,y)$-cylinders in $\Gamma$ with $(y,z)$-cylinders in $\Gamma$. Therefore $$\partial_y^-(\psi_y({\bf s}\otimes{\bf t}))=\psi_y(\partial_{xy}^-\otimes\partial_{yz}^-({\bf s}\otimes {\bf t})).$$
\end{proof}

\medskip

Theorem \ref{compthm} has many nice consequences. First, it implies that $\partial_y^-\circ\partial_y^-=0$. It also implies that $\partial_y^-$ lowers Maslov grading by one and keeps Alexander filtration level constant. In light of Theorem \ref{MOS}, one concludes that the filtered chain homotopy type of $(C_y^-(\Gamma),\partial_y^-)$ is a link invariant:

\medskip

\begin{corollary}
Let $\Gamma$ be a cube diagram representing $L$. Define the homology of $(C_y^-(\Gamma),\partial_y^-)$ to be $CH^-(L)=H_*(C_y^-(\Gamma),\partial_y^-)$. Then $CH^-(L) \cong HFK^-(L) \ot HFK^-(L)$. Moreover, since this homology only depends on $L$, we have
$$H_*(C_y^-(\Gamma),\partial_y^-)\cong H_*(C_z^-(\Gamma),\partial_z^-)\cong H_*(C_x^-(\Gamma),\partial_x^-).$$
\end{corollary}

The filtered chain homotopy type of $(C_y^-(\Gamma),\partial_y^-)$ can be checked to be a link invariant directly using Corollary~\ref{Moves_Corollary}.  Cube stabilizations and cube commutations correspond to grid stabilizations and grid commutations in the two chosen projections.

\medskip

If each of the $\mathtt{\bf X}_i$ and $\mathtt{\bf Y}_i$ variables are set to $0$, then we have the following corollary.
\begin{corollary}
Let $\Gamma$ be a cube diagram representing $L$ of size $n$ and
$$\widehat{CH}(L,n)=H_*(C_y^-(\Gamma)/\{\mathtt{\bf X}_i=\mathtt{\bf Y}_i=0\}_{i=1}^n).$$
Then
$$\widehat{CH}(L,n)\cong (\widehat{HFK}(L)\otimes V^{\otimes(n-1)})\otimes (\widehat{HFK}(L)\otimes V^{\otimes(n-1)}),$$
where $V$ is a $2$-dimensional vector space spanned by generators in bigradings $(-1,-1)$ and $(0,0)$.
\end{corollary}

\section{Acknowledgements}
The authors  thank Brendan Owens for helpful discussions.  We would also like to thank the referee for helping us improve the exposition of our paper. S. Baldridge was partially supported by NSF grant DMS-0507857 and NSF Career Grant DMS-0748636.  A. Lowrance was partially supported by NSF VIGRE grant DMS-0739382.

\end{document}